\documentclass[12pt]{amsart}

\usepackage[usenames,dvipsnames,svgnames,table]{xcolor}

\usepackage{xargs,url,hyperref}
\usepackage[colorinlistoftodos,prependcaption,textsize=tiny,linecolor=red,backgroundcolor=red!25,bordercolor=red]{todonotes}
\newcommandx{\kolja}[2][1=]{\todo[linecolor=brown,backgroundcolor=red!25,bordercolor=red,#1]{#2 ---Kl}}
\newcommandx{\caro}[2][1=]{\todo[linecolor=purple,backgroundcolor=brown!25,bordercolor=brown,#1]{#2 ---C}}
\newcommandx{\todocite}[2][1=]{\todo[linecolor=green,backgroundcolor=green!25,bordercolor=green,#1]{cite: #2}}

\newtheorem{theorem}{Theorem}
\newtheorem{proposition}[theorem]{Proposition}
\newtheorem{observation}[theorem]{Observation}
\newtheorem{question}[theorem]{Question}
\newtheorem{corollary}[theorem]{Corollary}
\newtheorem{lemma}[theorem]{Lemma}
\theoremstyle{definition}
\newtheorem{remark}[theorem]{Remark}
\newtheorem{example}[theorem]{Example}
\newtheorem{definition}[theorem]{Definition}
\newtheorem{conjecture}[theorem]{Conjecture}
\newtheorem{problem}[theorem]{Problem}

\usepackage{amssymb}
\usepackage{amsmath}
\usepackage{amsthm}
\usepackage{amsfonts}
\usepackage{amscd}
\usepackage{mathtools}
\usepackage{mathrsfs}
\usepackage[mathscr]{eucal}

\usepackage{listings, lstautogobble}
\lstnewenvironment{sagecode}{}{}

\usepackage{vmargin}
\setpapersize{USletter}
\setmargrb{2.5cm}{2cm}{2.5cm}{2.cm}
\usepackage{version}
\usepackage{cite}
\excludeversion{sagecode}

\usepackage[inline]{enumitem}

\usetikzlibrary{arrows,shapes}

\newcommand{\Gr}{\mathrm{Gr}}

\newcommand{\rank}{\operatorname{rank}}

\definecolor{cof}{RGB}{219,144,71}
\definecolor{pur}{RGB}{186,146,162}
\definecolor{greeo}{RGB}{91,173,69}
\definecolor{greet}{RGB}{52,111,72}

\newcommandx\conv{\operatorname{conv}}

\title[LPMs and quotients]{Lattice path matroids and quotients}

\author[Benedetti, Knauer]{Carolina Benedetti-Vel\'asquez \and Kolja Knauer}

\address[C. Benedetti]{
Departamento de Matemáticas, Universidad de los Andes, Bogotá, Colombia. Email: \href{mailto:c.benedetti@uniandes.edu.co}{c.benedetti@uniandes.edu.co}}

\address[K. Knauer]{
  Departament de Matem\`atiques i Inform\`atica -- Universitat de Barcelona
  \& Centre de Recerca Matemàtica -- Spain \\
  LIS, Aix-Marseille Universit\'e, CNRS, and Universit\'e de Toulon -- France.
  Email: \href{mailto:kolja.knauer@ub.edu}{kolja.knauer@ub.edu}
}
    
\begin{document}

\begin{abstract}
We characterize the quotients among lattice path matroids (LPMs) in terms of their diagrams. This characterization allows us to show that ordering LPMs by quotients yields a graded poset, whose rank polynomial has the Narayana numbers as coefficients. 

Furthermore, we study full lattice path flag matroids and show that -- contrary to 
arbitrary positroid flag matroids -- they correspond to points in the nonnegative flag variety. At the basis of this result lies an identification of certain intervals of the strong Bruhat order with lattice path flag matroids. 

A recent conjecture of Mcalmon, Oh, and Xiang states a characterization of quotients of positroids. We use our results to prove this conjecture in the case of LPMs.

\end{abstract}

\maketitle
\section{Introduction}

Matroids, introduced independently by Whitney~\cite{W34} and Nakasawa~\cite{NK09}, around 1930, are an abstraction of the concept of linear independence from linear algebra, carried to other settings such as graphs, systems of distinct representatives, transcendental extensions of fields, etc. This paper focuses on a class of matroids called \emph{representable} as defined in Section~\ref{subsec:matroids}. The family of representable matroids we are particularly interested in are  \emph{positroids}. Positroids appear in the work of da Silva from the perspective of oriented matroids (see \cite{dS87}, \cite{ARW17}), then by Blum~\cite{B07} in terms of Koszulness of rings associated to a matroid. Finally, Postnikov~\cite{P06} introduced positroids via a stratification of the totally nonnegative Grassmannian. This latter point of view is  the one that has spiked most of the research related to positroids, in particular, since part of the work of Postnikov includes several combinatorial characterizations of them. 

A categorical view point on matroids leads to the notion  of quotients, see~\cite{HP18,H68,Bry-86}. Matroid quotients are part of standard text books such as~\cite{O11} and natural appearances can be found in linear algebra and graph theory.
For instance, out of a graph one can construct a quotient after identifying some vertices. 
Despite the several ways that there are to define quotients, it can be very difficult to determine the quotients of a general matroid, and even worse, to characterize quotients for a given family of matroids.

The present paper focuses on a family of positroids called \emph{lattice path matroids}, LPMs for short. We provide an answer to the question: \begin{quote}
 Given two lattice path matroids $M$ and $M'$ on the same ground set, how can we determine combinatorially if $M$ is a quotient of $M'$?
                                                                                                                                                                                                                                                    \end{quote}
 
Any LPM  can be thought of as a diagram in the plane grid as in Figure~\ref{fig:xmpl}. Such a diagram is bounded above by a monotone lattice path $U$ and below by a non-crossing monotone lattice path $L$. Any monotone lattice path from the bottom left to the upper right corner inside this diagram is identified with a set $B$, where $i\in B$ if and only if the $i$th step of the path is  North. Now, the set $\mathcal{B}$ of these sets forms the set of bases of a matroid called the LPM $M[U,L]$. As a special case, a matroid is \emph{Schubert} if it is an LPM $M[U,L]$, where $U$ does all its North steps first. In particular, \emph{uniform} matroids are LPMs where furthermore $L$ does all its North steps last. Compare this with the definition of LPMs in terms of the Gale order, see Definition~\ref{def:lpm}. 

\begin{figure}[htp]
    \centering
    \includegraphics[width=.3\textwidth]{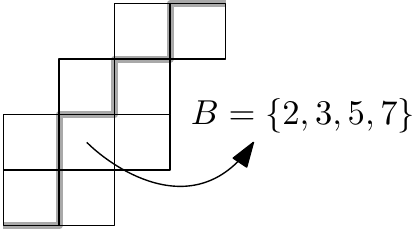}
    \caption{A basis in the diagram representing the LPM $M[1246,3568]$.}\label{fig:xmpl}
\end{figure}

LPMs were introduced by Bonin, de Mier, and Noy~\cite{Bon-03}, where fundamental properties were established. Many different aspects of lattice path matroids have been studied:  excluded minor characterizations~\cite{Bon-10}, representations over finite fields~\cite{padro2023efficient}, algebraic geometry notions~\cite{Del-12,Sch-10,Sch-11}, the Tutte polynomial~\cite{Bon-07,KMR,Mor-13}, the associated basis polytope in connection with its facial structure~\cite{An-17,Bid-12}, specific decompositions in relation with Lafforgue's work
~\cite{Cha-11,BKV23}, as well as its Ehrhart polynomial~\cite{Kna-18,Bid-12,BKV23,FJS22,Fer22}. 

The study of LPMs as a subclass of positroids, including analyzing quotients of these, is mostly novel apart from~\cite{DeM-07}, where certain quotients of LPMs related to the tennis ball problem are explored.  

One of the main contributions of this paper provides a way to determine all the quotients of a given LPM (Theorem~\ref{thm:lpm_quotients}). The advantage of this characterisation is that it allows to tell the quotients of an LPM purely based on its diagram. As a consequence of this result, we are able to build a graded poset $\mathcal P_n$ whose elements are LPMs ordered by quotients. Some enumerative results regarding $\mathcal P_n$ are stated in Corollary~\ref{cor:unimodular}, where it is shown that the rank function of $\mathcal P_n$ has as coefficients the  Narayana numbers.

A maximal sequence of distinct matroids on the same ground set, where each matroid is a quotient of the next, is a \emph{full flag matroid}, see~\cite{BGW}.
We can view full flag matroids consisting of LPMs as maximal chains in $\mathcal P_n$. Our interest in these flags, called lattice path flag matroids (LPFMs), arises from thinking of LPMs as positroids. See Section~\ref{sec:background} for the necessary background and motivation.

Positroids can be thought of as cells of the nonnegative Grassmannian. On the other hand, points in  the nonnegative flag variety $\mathcal F\ell _n^{\geq0 }$ can be thought of as certain full positroid flag matroids (PFMs)~\cite{TW15,KW15}. However, not every PFM arises this way (see Example~\ref{xmpl:postroids}). Moreover, in~\cite{TW15,KW15} the authors prove 
that the points in $\mathcal F\ell _n^{\geq0 }$  correspond to intervals in the (strong) Bruhat order. Our second main result shows that every LPFM corresponds to an interval in the Bruhat order and thus, a point in $\mathcal F\ell _n^{\geq0 }$ (Theorem~\ref{thm:iso_posets} and Corollary~\ref{cor:lpmint_bruhatint}). Moreover we characterize those intervals in the Bruhat order that come from LPFMs (Theorem~\ref{thm:quotients_asgoodpairs}). In particular, Proposition~\ref{thm:cubes_are_lpm} shows that cubes in the (right) weak Bruhat order are instances of these intervals.

%


Combining our description of LPM quotients with the fact that LPFMs are points in $\mathcal F\ell _n^{\geq0 }$ we achieve our final result Theorem~\ref{thm:suho}: the (realizable) quotient relation among LPMs can be expressed in terms of certain objects called CCW arrows in~\cite{MOX19}. This confirms a conjecture of Mcalmon, Oh, Xiang in the case of LPMs.

We finish with some structural questions on the poset structure of the set of LPMs ordered by quotiens, diagram representations of LPFMs as suggested by de Mier~\cite{DeM-07}, and Higgs lifts and the weak order on LPMs.

%
%
%

\section{Preliminaries}\label{sec:background}

%
%

\subsection{Matroids, positroids and the (real) Grassmannian}\label{subsec:matroids}

There are several equivalent ways to define matroids, see~\cite{O11}. For our purposes we say that a \emph{matroid} $M=(E,\mathcal{B})$ is a pair consisting of a finite set $E$ and a non-empty collection  $\mathcal{B}$ of subsets of $E$ that satisfies:
\begin{center}
if $A,B\in \mathcal{B}$ and $a\in A\setminus B$, then there is $b\in B\setminus A$ such that $(A\setminus \{a\})\cup \{b\}\in \mathcal{B}$.
\end{center}
In this context, we refer to the set $E$ as the ground set of $M$ and the collection $\mathcal B$ as the \emph{set of bases} of $M$. Also, an element $A\in\mathcal B$ is said to be a \emph{basis} of $M$.
Since the set $E$ has cardinality $n$, for some $n\geq 0$, we will identify it with the set $[n]:=\{1,\dots,n\}$. The \emph{uniform matroid} of rank $k$ over $[n]$, denoted $U_{k,n}$, is the matroid whose bases are all the subsets of size $k$ of $[n]$.

Given a matroid $M=([n],\mathcal B)$, it is known that elements of $\mathcal{B}$ have all the same cardinality, say $k\geq 0$, just as bases of a finite dimensional vector space have the same size.  In this case, we say that \emph{the rank} of $M$ is $k$, and we denote this as $r(M)=k$. A matroid $M=([n],\mathcal B)$ of rank $k$ is said to be \emph{representable} (over $\mathbb R$) if there exists a collection of vectors $S=\{ u_1,\dots, u_n\}\subseteq \mathbb R^k$ such that   $\dim(span(S))=k$ and $\{ i_1,\dots,i_k\}\in\mathcal B$ if and only if $\{ u_{i_1},\dots,u_{i_k}\}$ is a basis of $span(S)$. In this case, the $k\times n$ matrix whose columns are the set $S$ is said to be a \emph{(matrix) representation of $M$}. Although almost all matroids are non representable~\cite{N18}, in this paper the matroids we are interested in are the ones that are representable over $\mathbb R$. We will in the following elaborate on one of the many reasons why this class is important.


The (real) \emph{Grassmannian $\Gr_{k,n}$} consists of all the $k$-dimensional vector subspaces $V$ of $\mathbb R^n$. Let $V\in \Gr_{k,n}$ and let $\{v_1,\dots,v_k \}$ be a basis of $V$. Then the $k\times n$ matrix $A_V$ whose rows are $\{v_1,\dots,v_k \}$ gives rise to a representable matroid $M=([n], \mathcal B)$ of rank $k$ such that $B=\{ i_1,\dots,i_k\}\in\mathcal B$ if and only if $\Delta_B\neq 0$, where $\Delta_B$ is the $k\times k$ determinant of $A_V$ obtained from the columns indexed by the set $B$. 
Now let us talk about the \emph{nonnegative Grassmannian $\Gr_{k,n}^{\geq 0}$}. As a set, $\Gr_{k,n}^{\geq 0}$ consists of those $V\in \Gr_{k,n}$ for which there exists a full rank $k\times n$ matrix $A_V$, whose rows span $V$, such that every maximal minor of $A_V$ is nonnegative. The representable matroid $M=([n], \mathcal B)$ arising from such $V\in\Gr_{k,n}^{\geq 0}$, as explained before, is exactly what is called a \emph{positroid}. Note that for all maximal minors to be nonnegative, the ordering of the columns is essential, which is why a positroid is a matroid on $[n]$, where the (natural) ordering of the ground set is part of the input. Let us clarify this with an: 

\begin{example}
The matroid $P= ([4],\mathcal B)$ where $\mathcal B=\{13,14,23,24 \}$, is a positroid, since the matrix  $A_V=\begin{pmatrix}1&1&0&0\\0&0&1&1  \end{pmatrix}$ is such that each of the maximal minors indexed by the sets  $\{13,14,23,24 \}$ is positive and the remaining maximal minors are 0.  Notice that we are writing $ij$ to denote the subset $\{i,j\}$, as long as there is no confusion.  In particular, this example allows us to conclude that the subspace $V=span\langle (1,1,0,0), (0,0,1,1) \rangle$ is an element of $\Gr_{2,4}^{\geq 0}$. On the other hand, the matroid $M=([4],\mathcal B)$ where $\mathcal B=\{12,14,23,34 \}$ is representable but is not a positroid. We leave this as an exercise to the reader. Notice that the matroid $M$ corresponds to a relabelling of the elements of $P$, thus as remarked above being a positroid depends strongly on the ordering of the ground set. That is, being a positroid is in general not preserved under matroid isomorphisms.
\end{example}

We already mentioned several instances where positroids have appeared. For our purposes, the importance of positroids is that they contain the family of lattice path matroids, as will be defined in Section \ref{sec:LPMs}.
Although our treatment in the present paper is purely combinatorial, we want to emphasize that our initial interest for developing this project started from the connection between geometry and matroid theory via the Grassmanian (and its relatives), and representable matroids. 

Going back to our discussion above, let us scratch the surface of the connection that interests us between geometry and matroid theory. Several decompositions of the Grassmannian have been studied and many of them give rise to different families of representable matroids. In order to mention them we will denote by ${[n] \choose k}$ the collection of subsets $A\subset [n]$ such that $|A|=k$. 


\begin{definition}\label{def:i_gale}  
Let $A, B\in {[n] \choose k}$. We say that $A$ \emph{is smaller than }$B$ in the {Gale order} if, for every $r$ it holds that $a_r\leq b_r$, where $A=\{a_1<\cdots<a_k \}$ and $B=\{b_1<\cdots<b_k \}$. We denote this by by $A\leq_G B$.
\end{definition}

 We have discussed the cells of the Grassmanian and the particular positroid cells. Let us present two further specializations of positroid cells in terms of the Gale order. Note that these are equivalent with our definition from the Introduction and Definition~\ref{def:lpm}.

\begin{itemize}
\item \emph{Schubert cell} $\Omega_I$: Let $I\in {[n] \choose k}$. A generic point $U\in\Omega_I$ gives rise to a representable matroid $M_I=([n], \mathcal B)$ such that $B\in\mathcal B$ if and only if $I\leq_G B$. We call the matroid $M_I$ a Schubert matroid. For example, the matroid $M=([4],\{13,14,23,24,34 \})$ arises from the generic point $A=\begin{pmatrix}1&\star&\star&\star\\0&0&1&\star  \end{pmatrix}\in\Omega_{13}$, where the $\star$'s are generically chosen real numbers. That is, every pair of columns of $A$, except for $12$, is a basis of the column space of $A$.
\item \emph{Richardson cell} $\Omega_I^J$:  Let $I,J\in {[n] \choose k}$ such that $I\leq_G J$.  A generic point $U\in\Omega_I^J$ gives rise to a representable matroid $M_I^J=([n], \mathcal B)$ such that $B\in\mathcal B$ if and only if $I\leq_G B\leq_G J$. A matroid $M_I^J$ arising this way is known as a \emph{lattice path matroid} which will be denoted $M[I,J]$. In particular, every Schubert matroid is a lattice path matroid. For example, the Schubert matroid $M$ given above comes from a generic point in $\Omega_{13}^{34}$.  
\end{itemize}

\begin{remark}
Our definition of Schubert matroids is not closed
under isomorphism since it dependes on an ordering of the ground set. Definitions that do not depend on the ordering can be found in~\cite[Definition 7.5]{EHL23} and~\cite[Definition 2.20]{BS22}
These and isomorphic matroids are also known as “generalized Catalan matroids”, “shifted
matroids”, “nested matroids” and “freedom matroids” (see the discussion in~\cite[Section
4]{Bon-06}.
\end{remark}

The main goal of this paper is to describe combinatorially \emph{quotients} of lattice path matroids as will be defined shortly. The link with the previous discussion will be made via the \emph{flag variety}.

\subsection{Quotients of matroids and flag matroids}

As it goes with many concepts in matroid theory, the concept of quotient of matroids has many equivalent definitions. The interested reader is encouraged to consult, for instance, \cite{am2018flag,Bry-86,Kun86}. The definition we provide here is as follows.

\begin{definition}\cite[Prop. 7.4.7]{Bry-86}\label{def:quotient}
Consider two matroids $M$ and $M'$ on the ground set $[n]$ with base sets $\mathcal B$ and $\mathcal B'$, respectively. We say that $M'$ is a \emph{quotient} of $M$ if  for all $B\in\mathcal{B}, p\notin B$ there is $B'\in\mathcal{B'}$ such that $B'\subseteq B$ and if $B'\cup\{p\}\setminus\{q\}\in\mathcal{B'}$ then $B\cup\{p\}\setminus\{q\}\in\mathcal{B}$, for all $q\in B'$. We denote this by $M'\leq_Q M$.
\end{definition}

For example, as the reader may check we have that $U_{r,n}\leq_Q U_{s,n}$ for all $0\leq r\leq s\leq n$. Moreover, in \cite{BCT} the authors give a combinatorial way to determine some families of positroids that are a quotient of $U_{k,n}$, for any $0\leq k\leq n$. 
Observe that if $M'\leq_QM$ then $r(M')\leq r(M)$. In particular, $ r(M)= r(M')$ implies $M=M'$. On the other hand, Definition \ref{def:quotient} can be restated as follows. 
\begin{lemma}\label{lem:quotient_def}
Consider two matroids $M$ and $M'$ on the ground set $[n]$ with base sets $\mathcal B$ and $\mathcal B'$, respectively. Given $B\in\mathcal B$ and $p\in[n]\setminus B$ we set 
\begin{equation}\label{eq:f_circ}
    B_p:=\{q\in B\mid B+p-q\in\mathcal{B}\}.
\end{equation}
 Then we obtain that $M'\leq_Q M$ if and only if for all $B\in\mathcal{B}$ and $p\in[n]\setminus B$ there is $B'\in\mathcal{B'}$ such that $B'\subseteq B$ and $B'_p\subseteq B_p$. 
\end{lemma}
Although Definition \ref{def:quotient} seems tricky to work with, as one may suspect, the notion of matroid quotient is better understood for certain families of matroids. For example, for $k\leq n$, given a full rank $k\times n$ matrix $A$ let $M_A$ be the realizable matroid on $[n]$ of rank $k$ that $A$ gives rise to. Now let $A'$ be the $i\times n$ submatrix obtained from $A$ by deleting its bottom $k-i$ rows, for some $i\in[k-1]$. Then the representable matroid $M_{A'}$ that $A'$ gives rise to, is a quotient of the matroid $M_A$. What we are interested in is a handy and combinatorial way to determine when two lattice path matroids $M'$ and $M$ on $[n]$ are such that $M'\leq_Q M$. Note that $B_p$ is the fundamental circuit of the pair of basis $B$ and element $p$ and that circuits of LPMs where characterizaed in \cite[Theorem 3.9]{Bon-06}, so in view of Lemma~\ref{lem:quotient_def} there might be another approach. However, we pursue a differented strategy. In fact, we care about giving a characterisation of \emph{flags} of LPMs. Flag matroids, their polytopes and the positivity of such flags are combinatorial objects which are in focus of mathematicians and physicists at the moment, see in particular~\cite{Bor22,BEW22,black2022flag,corey2022initial,jarra2022flag,Joswig_2023}.

\begin{definition}[\cite{BGW}]\label{def:flag_matroid}
A \emph{flag matroid} is a sequence $\mathcal F=(M_0,M_1,\ldots,M_k)$  of distinct matroids on the ground set $[n]$ such that $M_i$ is a quotient of $M_{i+1}$ for $i\in \{0,1,\dots,k-1 \}$. 
If $k=n$, then we say that $\mathcal F$ is a \emph{full} flag matroid. Each of the $M_i$'s is said to be a \emph{constituent of $\mathcal F$}. If $B_0\subseteq\cdots \subseteq B_k$ is a sequence where $B_i$ is a basis of $M_i$, we refer to it as a \emph{flag of bases in $\mathcal F$}.
If every $M_i$ is a positroid we say that  $\mathcal F$ is a  \emph{positroid flag matroid} (PFM). If every $M_i$ is an LPM  we say that  $\mathcal F$ is a   
\emph{lattice path flag matroid} (LPFM).
\end{definition}

From Definition \ref{def:flag_matroid} we remark that if $\mathcal F=(M_0,M_1,\ldots,M_k)$ is a flag matroid then $r(M_0)<r(M_1)<\cdots <r(M_k)$. In particular if the flag $\mathcal F$ is a full flag, then $M_0$ is the matroid $U_{0,n}$ and $M_n=U_{n,n}$.  For our purposes, we will only focus on full flag matroids either if in the PFM or the LPFM case.

Now we want to extend the dictionary between   $\Gr_{k,n}$ and representable matroids, given so far. 
The \emph{(real) full flag variety $\mathcal F\ell_n$} consists of sequences (flags) of vector spaces $\mathcal F: \{\mathbf{0}\}=V_0\subset V_1\subset V_2\subset\cdots\subset V_n=\mathbb R^n$ such that $V_i\in \Gr_{i,n}$ for $i=1,\dots, n$. Thus each such $\mathcal F$ can be thought of as a full rank $A_{n\times n}$ matrix whose top $j$ rows give rise to a representable matroid $M_j$ of rank $j$. Therefore, the point  $\mathcal F\in\mathcal F\ell_n$ gives rise to the full flag matroid $\mathcal F=(M_0,M_1,\ldots,M_n)$.
In this case $\mathcal F$ is said to be a \emph{representable flag matroid} (over $\mathbb{R}$), and $A$ \emph{represents} the flag matroid $\mathcal F$.
However, even if two representable matroids $M$ and $M'$ are such that $M'\leq_Q M$, they do not necessarily form (part of) a representable  flag matroid. This is, there may be no matrix $A$ that gives rise to both of them, simultaneously (see~\cite[Section 1.7.5]{BGW} or~\cite[Example 6.9]{am2018flag}).

Finally, the \emph{nonnegative full flag variety $\mathcal F\ell_{n}^{\geq 0}$} consists of sequences  $\mathcal F: \{0\}=:V_0\subset V_1\subset V_2\subset\cdots\subset V_n=\mathbb R^n$ of vector spaces such that $\mathcal F$ can be given by a full rank $A_{n\times n}$ matrix whose top $j$ rows span $V_j$ as a point in $\Gr_{j,n}^{\geq 0}$ for each $j\in[n]$. That is, $A$ is such that each submatrix $A_j$ has nonnegative maximal minors and its row-space spans $V_j$, for each $j\in[n]$. In this case we say that $\mathcal F$ is \emph{nonnegatively representable}.
The following problems are in order:

\begin{itemize}
\item[(P1)] Does every full positroid flag matroid $\mathcal F$ come from a point in $\mathcal F\ell_{n}^{\geq 0}$? 
\item[(P2)] Does every full lattice path flag matroid $\mathcal F$ come from a point in $\mathcal F\ell_{n}^{\geq 0}$? 
\item[(P3)] Can we describe the family of flag matroids coming from points in $\mathcal F\ell_{n}^{\geq 0}$?
\end{itemize}


From now on when we refer to a flag matroid (of any kind) we mean a full flag matroid. Thus, LPFMs refer to full flags of LPMs, and similarly for PFM. Now, if the answer to problem P1 were affirmative, then P2 would be as well, since the family of  LPFMs is a subset of the family of  PFMs.  The discussion we have conveyed here is summarized in  Table~\ref{table:grass_matroid}. 

In this paper, we will see that the answer to problem P2 is yes. Now let us illustrate why the answer to P1 is negative. This makes P3 relevant as one may be misled into thinking that points in $\mathcal F\ell_{n}^{\geq 0}$? are precisely flags of positroids.

\begin{table}[h]
\begin{tabular}{|c|c|}
\hline
\textbf{Geometry} & \textbf{Matroids }  \\
\hline
\hline
Point $V$ in $\Gr_{k,n}$ & Representable matroid $M=([n],\mathcal B)$ of rank $k$\\
\hline
Richardson cell $\Omega_I^J$ & Lattice path matroid $M[I,J]$\\
\hline
Point $V$ in $\Gr^{\geq 0}_{k,n}$  & Positroid $M=([n],\mathcal B)$ of rank $k$\\
\hline
\hline
Flag $F:V_0\subset\cdots\subset V_n$ in $\mathcal F\ell_n$  & Representable flag matroid $M_0\leq_Q\cdots\leq_Q M_n$\\
\hline

Flag $F:V_0\subset\cdots\subset V_n$ in $\mathcal F\ell^{\geq 0}_n$&  (P3)\\
\hline

(P2) & lattice path flag matroid $M_0\leq_Q M_1\leq_Q\dots \leq_Q M_n$\\
\hline
(P1)  & positroid flag matroid $M_0\leq_Q M_1\leq_Q\cdots\leq_Q M_n$\\
\hline
\end{tabular}
\caption{\label{table:grass_matroid}Bridge between geometry and realizable matroids.}
\end{table}

\begin{example}\label{xmpl:postroids} 
Let $M_1$ be the positroid on $[3]$ whose set of bases is $\mathcal B_1=\{ 1,3\}$ and let $M_2=U_{2,3}$ be the uniform matroid of rank 2 on $[3]$. That is, the bases of  $M_2$ are $\mathcal B_2=\{12,13,23 \}$. We leave to the reader the task to check that $M_1$ and $M_2$ are positroids\footnote{In fact every uniform matroid is a positroid.} and that $M_1\leq_Q M_2$. Thus the flag $\mathcal F:U_{0,3}\leq_Q M_1\leq_Q M_2\leq_Q U_{3,3}$ is a PFM. If $\mathcal F$ came from an element in $\mathcal F\ell_n^{\geq 0}$ then there would be a $3\times 3$ matrix  $$
A=\begin{pmatrix}
a&0&b\\
c&d&e\\
f&g&h.
\end{pmatrix}$$ such that $\det A>0$ and also the submatrices $\begin{pmatrix}a&0&b\end{pmatrix}$ and $\begin{pmatrix}a&0&b\\c&d&e\end{pmatrix}$ would be a representation of the positroids $M_1$ and $M_2$, respectively. This forces $a>0$ and $b>0$ since $\mathcal B_1=\{ 1,3\}$. On the other hand, since $12\in\mathcal B_2$ then $ad>0$ and thus $d>0$. Similarly, since $23\in\mathcal B_2$ then $-bd>0$ and thus $d<0$ which is a contradiction. Thus, we are not able to obtain the PFM $\mathcal F:U_{0,3}\leq_Q M_1\leq_Q M_2\leq_Q U_{3,3}$ as coming from a point in $\mathcal F\ell_n^{\geq 0}$. 
\end{example}

It is known that every uniform matroid $U_{k,n}$ is an LPM and, as mentioned above, in \cite{BCT} the authors give a partial characterization of positroids $M$ such that $M\leq_Q U_{k,n}$. Here, we are interested in particular in a description of those LPMs $M'$ such that given an LPM $M$ it follows that $M'\leq_Q M$. Thus, a complete answer to this question, which we will give, does not imply the aforementioned result in \cite{BCT} since some quotients of $U_{k,n}$ are not LPMs.

To our knowledge it is open whether every positroid flag matroid corresponds to a point in the flag-variety. 

\section{Quotients of LPMs}\label{sec:LPMs}
Let $B',B\in {[n] \choose k}$. We say that $B'$ \emph{is smaller than }$B$ in the Gale order if $b_i'\leq b_i$, for all $i\in[k]$, where $B'=b'_1<\cdots< b'_k$ and $B=b_1<\cdots< b_k$, for some $k\leq n$. We denote this by $B'\leq_G B$.
In view of this, let us recall the definition of lattice path matroid. 

\begin{definition}\label{def:lpm}
Let $0\leq k\leq n$ and let $U,L\in {[n] \choose k}$ be such that $U\leq_G L$. The \emph{lattice path matroid $M[U,L]$} is the matroid over the set $[n]$ whose collection of bases is given by $\mathcal B=\{B\in {[n] \choose k}\,|\, U\leq_G B\leq_G L \}$. 
\end{definition}

Setting $M=M[U,L]$ in Definition \ref{def:lpm} it follows that $M$ has rank $k$. In particular $U$ and $L$ are bases of $M$. Generally we fix notation by setting $U=\{u_1<\cdots<u_k \}$ and $L=\{\ell_1<\cdots<\ell_k \}$. Then $U$ corresponds to the lattice path from $(0,0)$ to $(k,n-k)$ whose North steps are labelled by $U$, and similarly for $L$. Thus, if $B$ is any basis of $M$ then $B$ corresponds to a lattice path from $(0,0)$ to $(n-k,k)$ whose labels are in $B$ and $B$ lies between $U$ and $L$ since $U\leq_G B\leq_G L$ . See Figure~\ref{fig:xmpl}, where the basis $\{2,3,5,7\}$ of $M[1246,3568]$ is represented in the diagram.
The following is derived from the very definition of $\leq_G$.

\begin{observation}\label{obs:weak} Let $0\leq k\leq n$ and let $M[U,L]$ be an LPM of rank $k$ over $[n]$. The Gale order endows the set of bases of $M[U,L]$ with the poset structure of the interval $[U,L]_G$ of the bases of $U_{k,n}$ ordered by $\leq_G$. 
\end{observation}

Observation~\ref{obs:weak} in particular yields that ordering the bases of an LPM by $\leq_G$ endows the set $\mathcal{B}$ of bases with a distributive lattice structure, that has been characterized in~\cite{Kna-18}. See Figure~\ref{fig:lattice} for an example.

\begin{figure}[htp]
    \centering
    \includegraphics[width=.6\textwidth]{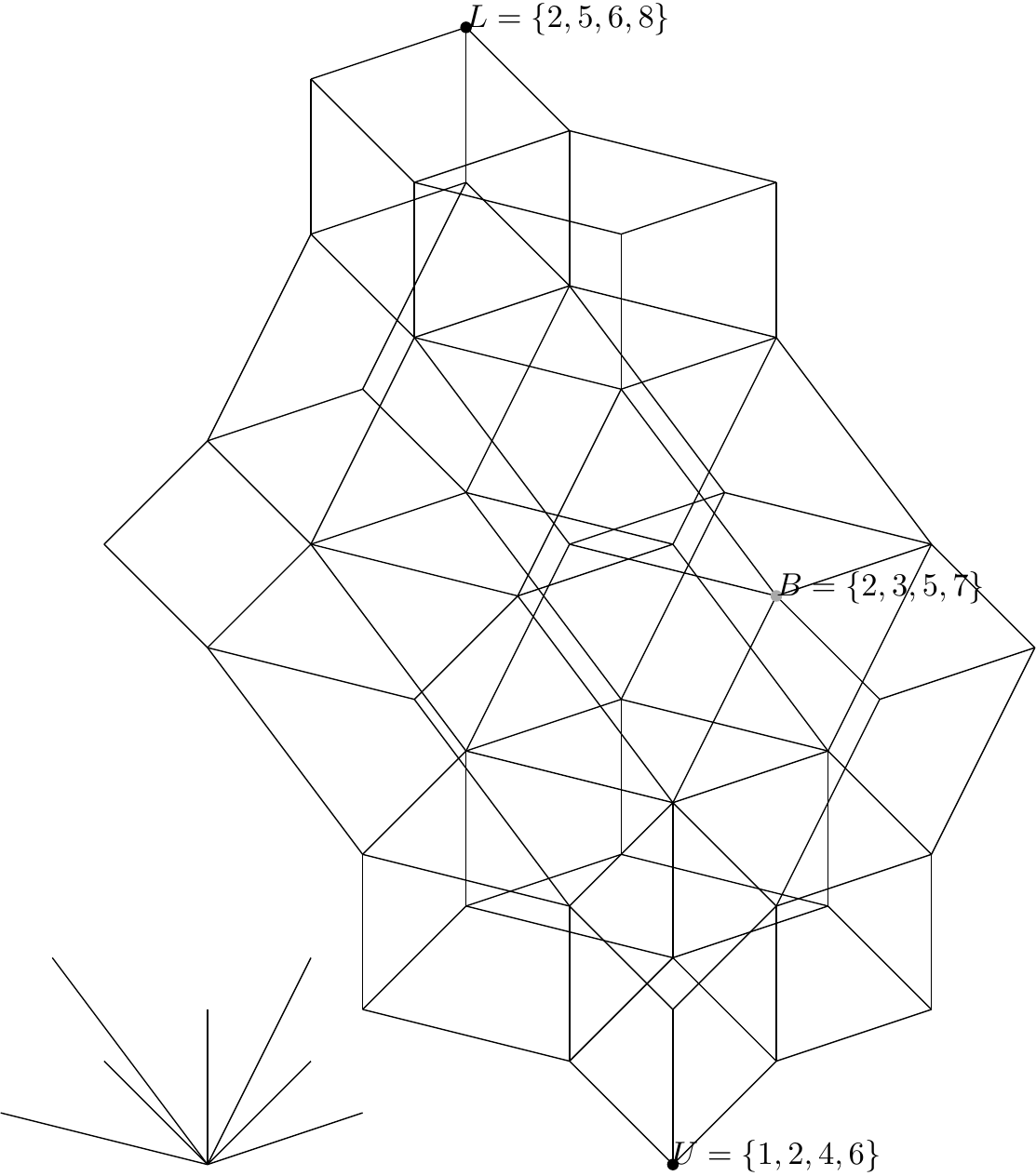}
    \caption{The lattice of bases of $M[1246,3568]$.}\label{fig:lattice}
\end{figure}

In what remains for this section we intend to describe combinatorially quotients of LPMs. In particular, we will determine when $M[U\setminus\{u\},L\setminus\{\ell\}]$ is a quotient of $M[U,L]$, where $u\in U$ and $\ell\in L$. Let us start by gathering some more intuition. Given $A\in{[n]\choose k}$ we denote its elements using lower case as $A=\{a_1<\cdots<a_k \}$.

If $M=M[U,L]$, it is not true in general that $M[U\setminus\{u_j\},L\setminus\{\ell_i\}]$ is a quotient of $M$, for any choice of $i,j\in[k]$. As an example  let  $M=M[1357,3578]$ and take the basis $B=1467$, also take $j=4$ and $i=1$. Using the notation from (\ref{eq:f_circ}) we have that  $B_5=\{q\in B\mid B+5-q\in M\}=\{4,6\}$. Moreover, for any $i\leq r\leq j$ one can check that in the matroid $M[135,578]$ we obtain  $(B\setminus\{b_r\})_5=B\setminus\{b_r\}\not\subseteq B_5$. 

We also point out that not every quotient of an LPM is an LPM. Indeed, \emph{every} matroid on $[n]$ is a quotient of the LPM $U_{n,n}$. As an another example, the \emph{($i$th-)truncation} of an LPM $M$, i.e., with base set given by $\mathcal{B}_i:=\{X\in{[n] \choose r-i}\mid \exists B\in\mathcal{B}: X\subseteq B\}$, is a quotient of $M$, although it may not be an LPM. A particular example of this situation comes from taking the LPM given by the direct sum $M=U_{1,2}\oplus U_{1,2}\oplus U_{1,2}$, see Figure~\ref{fig:u12x3}. Its first truncation is not an LPM, although it is a positroid, see~\cite{Bon-06}. 

\begin{figure}[htp]
    \centering
    \includegraphics[width=.15\textwidth]{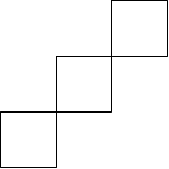}
    \caption{A diagram representing $M[135,246]=U_{1,2}\oplus U_{1,2}\oplus U_{1,2}$.}\label{fig:u12x3}
\end{figure}

The following will be essential for our results and is illustrated in  Figure~\ref{fig:Bp} as a visual aid.
\begin{lemma}\label{lem:Bp}
 Let $M=M[U,L] $ of rank $k$ and let $B=\{b_1<\cdots<b_k\}$ be a basis of $M$ and set $b_0=0$ and $b_{k+1}=n+1$. Let $p\in[n]$ such that $p\notin B$ and take the $x\in[k+1]$ such that $b_{x-1}<p<b_x$. Then $B_p=\{b_s< \ldots< b_t\}$ where 
 \begin{itemize}
 \item[(a)]   $1\leq s\leq x$ and $b_{s+1}\leq\ell_{s}, \ldots, b_{x-1}\leq\ell_{x-2}, p\leq \ell_{x-1}$,
 \item[(b)] $x-1\leq t\leq k$ and $p\geq u_x, b_x\geq u_{x+1}, \ldots, b_{t-1}\geq u_t.$
 \end{itemize}
\end{lemma}
\begin{proof}
 For the proof consider Figure~\ref{fig:Bp}, where the basis $B$ is a monotone path $P$ in the LPM diagram. Since $p\notin B$, it corresponds to a horizontal segment of $P$. Now, $B_p$ consists of those vertical segments $q$ of $P$ that can be made horizontal such that after making $p$ vertical, the path $Q$ corresponding to $B\setminus\{q\}\cup\{p\}$ remains within the boundaries of the diagram. These segments are (a) between the last time $B$ touched $L$ before arriving at $p$ and $p$ itself or (b) after $p$ and the next time $B$ touches $U$. This is what is expressed through indices in the statement of the lemma. 
\end{proof}

\begin{figure}[htp]
    \centering
    \includegraphics[width=.5\textwidth]{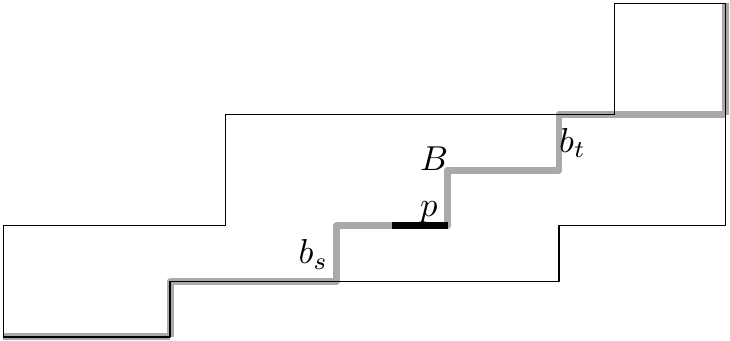}
    \caption{An illustration of Lemmas~\ref{lem:quotient_def} and~\ref{lem:Bp}: $B_p$ are the North steps in $B$ that can be turned into East steps, such that if $p$ is turned into a North step, then the resulting path is valid. They are precisely the North steps between $b_s$ and $b_t$.}\label{fig:Bp}
\end{figure}

\begin{definition}\label{def:good_pair}
Let $M=M[U,L]$ be an LPM where $U=\{u_1<\cdots< u_k\}$, $L=\{\ell_1<\cdots <\ell_k\}$. Let $1\leq i,j\leq k$. We say that $(\ell_i,u_j)$ is a \emph{good pair of $M$} if \begin{enumerate}
            \item $i\leq j$,
            \item $u_j-\ell_i\leq j-i$.
\end{enumerate}
Otherwise, we say that $(\ell_i,u_j)$ is a \emph{bad pair of $M$}. 
\end{definition}

We point out that Definition~\ref{def:good_pair} is equivalent to saying that $(\ell_i,u_j)$ is a good pair of $M$ if and only if $\max\{0,u_j-\ell_i\}\leq j-i$.
Graphically, being a good pair can be visualized as follows. The step $u_j$ is such that its northern vertex $(a,b)$ determines the closed region bounded below by $L$, and lies in the halfspaces $x\geq a$ and $y\leq b$. Then the pair  $(\ell_i,u_j)$ is a good pair if  $\ell_i$ lies in this region.  Figure~\ref{fig:nonquotient} depicts a bad pair $(\ell_i,u_j)$. Every good pair $(\ell_i,u_j)$ allows us to characterize LPMs of rank $k-1$ that are a quotient of a given LPM $M=M[U,L]$ of rank $k$ as the upcoming result (which will turn out to be an equivalence) shows.  

\begin{proposition}\label{prop:goodpair}
Let $M=M[U,L]$ be such that $r(M)=k$ and let $(\ell_i,u_j)$ be a good pair of $M$. Then the matroid $M'=M[U\setminus\{u_j\},L\setminus\{\ell_i\}]$ of rank $k-1$ is a quotient of $M$.
\end{proposition}
\begin{proof}
 Let $B\in\mathcal{B}(M)$ and let $(\ell_i,u_j)$ be a good pair of $M$ for some $1\leq i\leq j\leq k$. For any $r\in [k]$
it holds that $B\setminus\{b_r\}\subseteq B$. Furthermore, since $U\leq_G B\leq_G L$ and if $i\leq r\leq j$  we have $U\setminus\{u_j\}\leq_G U\setminus\{u_r\} \leq_G B\setminus\{b_r\}\leq_G L\setminus\{\ell_r\}\leq_G L\setminus\{\ell_i\}$. Therefore $B\setminus\{b_r\}$ is a basis of $M'$. 

Now let $p\notin B$. We will show that if $\max(0,u_j-\ell_i)\leq j-i$, then $r$ can be chosen such that $(B\setminus\{b_r\})_p\subseteq B_p$. We use the description of $B_p$ provided by Lemma~\ref{lem:Bp}.
%
%
We want to choose $i\leq r\leq j$ such that for $L'=L\setminus\{\ell_i\}$, $B'=B\setminus\{b_r\}$, $U'=U\setminus\{u_j\}$ and the correspondingly defined $s',t'$ we have that $b_s\leq b'_{s'}$ and $b_t\geq b'_{t'}$.

\emph{Case 1:} Let $i<s$ and $t<j$. In this situation it holds that $\ell_i\leq\ell_{s-1}<b_s<\cdots<p<\cdots<b_t<u_{t+1}\leq u_j$. Thus, $u_j-\ell_i>t-s+2\geq j-i$, which contradicts our assumption on $(\ell_i,u_j)$ being a good pair. Hence, we cannot have $i<s$ and $t<j$ simultaneously.

\emph{Case 2:} If $s>i$, then we set $r=i$. We get either   $\ell'_{s-2}\leq\ell_{s-1}<b_s=b'_{s-1}\leq b'_{s'}$ or $s-1=1\leq s'$. Since $r\leq j\leq t$, one can see that either $u'_{t}=u_{t+1}>b_t\geq b'_{t-1}\geq b'_{t'}$ or $t=r\geq t'$.

\emph{Case 3:} Similarly, if $t<j$, then we set $r=j$ and we obtain that either $u'_{t+1}\geq u_{t+1}>b_t=b'_t\geq b'_{t'}$ or $t=r\geq t'$. By the above we have $s\leq i\geq r$, we compute either $\ell'_{s-1}=\ell_{s-1}<b_s\leq b'_s\leq b'_{s'}$ or $s=1\leq s'$.

\emph{Case 4:} If $s\leq i$ and $t\geq j$ any choice of $i\leq r \leq j$ yields a good $B'$. Indeed, as above we will get $u'_{t}=u_{t+1}>b_t\geq b'_{t-1}\geq b'_{t'}$ and $\ell'_{s-1}=\ell_{s-1}<b_s\leq b'_s\leq b'_{s'}$.
\end{proof}

\begin{figure}[htp]
    \centering
    \includegraphics[width=.35\textwidth]{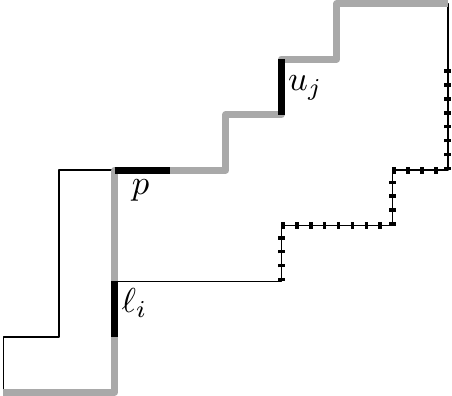}
    \caption{An LPM with a bad pair $(\ell_i, u_j)$ . The gray basis $B$ has $(B\setminus\{b_r\})_p\not\subseteq B_p$ for all $r$. Exactly those $\ell\in L$ on the dotted path yield good pairs with $u_j$.}\label{fig:nonquotient}
\end{figure}

\begin{lemma}\label{lem:containment} Let $M'=M[U',L']$ and $M=M[U,L]$.
 If $M'\leq_Q M$, then $U'\subseteq U$ and $L'\subseteq L$.
\end{lemma}
\begin{proof}
 We only show $U'\subseteq U$, the proof that $L'\subseteq L$ is analogous. Suppose, by contradiction, that  $U'\not\subseteq U$ and choose the smallest $p\in U'\setminus U$. By Lemma~\ref{lem:Bp} (also see Figure~\ref{fig:Bp}) we know that $U_p$ consists of all North steps in $U$ that can be made East in order to yield a valid path when $p$ is made North. Since $U$ is the upper path this yields $u<p$ for all $u\in U_p$.
 
 Now, following Lemma~\ref{lem:Bp} we take the $x$ such that $u_{x-1}<p<u_x$. Let now $B'=\{b_1< \cdots< b_k\}$ be a basis of $M'$ such that $B'\subseteq U$ and $B'_p\subseteq U_p$. Such $B'$ exists since $U$ is a basis of $M$ and $M'\leq_Q M$, by Lemma~\ref{lem:quotient_def}. Since  $p$ is the smallest element in $U'\setminus U$, we have $b_i\geq u_i$ for all $i<x$. Since $p\in U'$ and $B'\subseteq U$ we have $b_x>p=u_x$. Since $u<p$ for all $u\in U_p$, we have $b_x\notin U_p$. However, Lemma~\ref{lem:Bp} yields $b_x\in B'_p$, because $p$ is a North step in $U'$ but not in $B'$, but $b_x$ is the next North step in $B'$ after $p$. This leads to a contradiction with Lemma~\ref{lem:quotient_def}.
\end{proof}
%

\begin{lemma}\label{lem:greedyprimal}  Let $M'=M[U',L']$ and $M=M[U,L]$, where $U'=\{u'_1<\cdots< u'_{k'}\}$, $L'=\{\ell_1'<\cdots <\ell_{k'}'\}$, $U=\{u_1<\cdots< u_{k}\}$, $L=\{\ell_1<\cdots <\ell_{k}\}$. Denote $U\setminus U'=\{u_{i_1}<\cdots <u_{i_z}\}$ and $L\setminus L'=\{\ell_{j_1}<\cdots < \ell_{j_z}\}$.
If $M'\leq_Q M$, then $\{{j_1}<\cdots <{j_z}\}\leq_G\{{i_1}<\cdots <{i_z}\}$.
\end{lemma}
\begin{proof}
We argue by contradiction.  Suppose that $J:=\{{j_1}<\cdots <{j_z}\}\not\leq_G\{{i_1}<\cdots <{i_z}\}$ and let $w$ be the smallest index such that $i_w<j_w$. The choice of $w$ yields 
$j_1<\cdots <j_{w-1}\leq i_{w-1}<i_w<j_w$ and in particular $i_w\notin J$. Then we have $\ell_{i_w}=\ell_{i_w-w+1}'$. That is, the ${i_w}$-th North step of $L$ is also a North step of $L'$, but appears $w-1$ North steps earlier. Similarly, we have $u_{i_w}<u'_{i_w-w+1}$.
Consider now the set $B=\{u_1,\ldots,u_{i_w},\ell_{i_{w}+1},\ldots,\ell_k\}$, which is a basis of $M$, since one can view it as following first $U$, then passing all to the East until hitting $L$ and then continuing $L$ until the end. By the quotient relation there is a set $Z$ of size $z$ such that $U'\leq_G B'\leq_G L'$ where $B':=B\setminus Z$. 
 
 Now, since  $U'\leq_G B'$ we have $u_{i_w}<u'_{i_w-w+1}\leq b'_{i_w-w+1}$ which by the shape of $B$ implies  $b'_{i_w-w+1}\geq \ell_{i_w+1}$. With $\ell_{i_w+1}>\ell_{i_w}=\ell_{i_w-w+1}'$ this yields $b'_{i_w-w+1}>\ell_{i_w-w+1}'$ and contradicts $B'\leq_G L'$. 
 \end{proof}
 
 If $M=M[U,L]$ is an LPM on the ground set $[n]$ then its dual matroid $M^*$ is such that $M^*=M[\overline{L},\overline{U}]$ where $\overline A:=[n]\setminus A$ for $A\subseteq [n]$, see e.g.~\cite{Bon-06}. Then, Lemma \ref{lem:greedyprimal} can be stated in terms of $M^*$ and $M'^*$ since $M'\leq_Q M$ if and only if $M^*\leq_Q {M'}^*$ , see~\cite[Proposition 7.4.7]{Bry-86} and $U\setminus U'=\overline{U'}\setminus\overline{U}$. Thus, by duality we obtain the following result. We leave the details of the proof to the reader.

\begin{lemma}\label{lem:greedydual}
 Let $M'^*=M[\overline{L'},\overline{U'}]$ and $M^*=M[\overline{L},\overline{U}]$, where $\overline{U'}=\{\overline{u'}_1<\cdots< \overline{u'}_{n-k'}\}$, $\overline{L'}=\{\overline{\ell}'_1<\cdots <\overline{\ell}'_{n-k'}\}$, $\overline{U}=\{\overline{u}_1<\cdots< \overline{u}_{n-k}\}$, $\overline{L}=\{\overline{\ell}_1<\cdots <\overline{\ell}_{n-k}\}$. Denote $\overline{U'}\setminus\overline{U}=\{\overline{u'}_{i_1},\ldots, \overline{u'}_{i_z}\}$ and $\overline{L'}\setminus\overline{L}=\{\overline{\ell}_{j_1}',\ldots, \overline{\ell}_{j_z}'\}$.
If $M^*\leq_Q M'^*$, then $\{{i_1},\ldots, {i_z}\}\leq_G \{{j_1},\ldots, {j_z}\}$.
\end{lemma}

 The following definition is an extension of Definition \ref{def:good_pair}. Given LPMs $M'\leq_Q M$ it will allow us to provide a sequence $M_1,\dots,M_{z-1}$ of LPMs  of ranks $k-z+1,\dots,k-1$, respectively, such that $M'\leq_Q M_1\leq_Q\cdots\leq_Q M_{z-1}\leq_Q M$.

\begin{definition}[Pairings]\label{def:good_pairing}
Let $M=M[U,L]$ be an LPM where $U=\{u_1<\cdots< u_k\}$, $L=\{\ell_1<\cdots <\ell_k\}$. Let $\widetilde{U}=\{u_{i_1}<\cdots <u_{i_z}\}$ and $\widetilde{L}=\{\ell_{j_1}<\cdots < \ell_{j_z}\}$ be subsets of $U$ and $L$, respectively.
\begin{itemize}
    \item[(a)] Given $\pi:[z]:\rightarrow\{j_1<\cdots <j_z\}$  and $\psi:[z]\rightarrow\{i_1<\cdots <i_z\}$ bijections, the sequence $((\ell_{\pi(1)},u_{\psi(1)})),\dots, (\ell_{\pi(z)},u_{\psi(z)}))$ is called a \emph{pairing} of $(\widetilde{L}, \widetilde{U})$. 
    \item[(b)] A pairing is \emph{good} if  $(\ell_{\pi(r)},u_{\psi(r)})$ is a good pair of the LPM $M[U',L']$ where $U'=U\setminus\{u_{\psi(1)},\dots,u_{\psi(r-1)} \}$ and $L'=L\setminus\{\ell_{\pi(1)},\dots,\ell_{\pi(r-1)} \}$, for $1\leq r\leq z-1$.
    \item[(c)] A pairing is \emph{greedy} if $\pi$ and $\psi$ are order-preserving. That is, if it is of the form $((\ell_{j_1},u_{i_1}),\ldots, (\ell_{j_z},u_{i_z}))$. 
\end{itemize}  
\end{definition}



\begin{lemma}\label{lem:greedypairing}
Let $M'=M[U',L']$ and $M=M[U,L]$ be such that $M'\leq_Q M$ and $r(M)=k$. Let
$U\setminus U'=\{u_{i_1}<\cdots <u_{i_z}\}$ and $L\setminus L'=\{\ell_{j_1}<\cdots < \ell_{j_z}\}$. Then the greedy pairing $((\ell_{j_1},u_{i_1}),\ldots, (\ell_{j_z},u_{i_z}))$ of $(L\setminus L', U\setminus U')$ is good.
\end{lemma}
\begin{proof}
 Let $(\ell,u)$ be an element of the greedy pairing. We want to show that $(\ell,u)$ satisfies Definition \ref{def:good_pair}. By Lemma~\ref{lem:greedyprimal} it follows that $(\ell,u)=(\ell_{j_y},u_{i_y})$ for some $\ell_{j_y}\in L, u_{i_y}\in U'$ where $i_y\geq j_y$.

 Now, using Lemma~\ref{lem:greedydual} we have that $(\ell,u)=(\overline{\ell}_{j_r},\overline{u}_{i_r}) $ for some $\overline{\ell}_{j_r}\in\overline{L'}, \; \overline{u}_{i_r}\in\overline{U'} $ with $i_r\leq j_r$. Thinking of $L'$ as a lattice path, this means, that starting from $(0,0)$, there are as many east steps in $L'$ before $\ell$  as there are east steps before $u$ in $U'$. 
Then by the choice of the greedy pairing, we have that $\ell$ is (weakly) to the right of $u$ in $M$. We conclude that $(\ell,u)$ is good.
\end{proof}

The next result will be the remaining ingredient towards the proof of the main theorem in this section.
\begin{lemma}\label{lem:removing}
 Let $M=M[U,L]$, $\ell_i<\ell_{i'}\in L$ and $u_j<u_{j'}\in U$. If $(\ell_i,u_j)$ and $ (\ell_{i'},u_{j'})$ are good then $(\ell_{i},u_{j})$ is good in $M[U\setminus\{u_{j'}\},L\setminus\{\ell_{i'}\}]$ and $(\ell_{i'},u_{j'})$ is good in $M[U\setminus\{u_{j}\},L\setminus\{\ell_{i}\}]$.
\end{lemma}
\begin{proof}
The first statement follows since removal of $(\ell_{i'},u_{j'})$ does not change the positions of $(\ell_{i},u_{j})$. The second statement follows because the removal of $(\ell_{i},u_{j})$ shifts both segments $(\ell_{i'},u_{j'})$ one unit to the right and downwards, so if they were good before they are still good afterwards.
\end{proof}
Note that the condition of the comparability of the pairs is necessary (see Figure~\ref{fig:intervalP3}).
Now we are ready to state the main result of this section.

\begin{theorem}\label{thm:lpm_quotients}[Characterizing quotients of LPMs]
Let $M'=M[U',L']$  and $M=M[U,L]$ be LPMs on the ground set $[n]$. We have that $M' \leq_Q M$ if and only if $U'\subseteq U$, $L'\subseteq L$ and the greedy pairing of $(L\setminus L',U\setminus U')$ is good.
\end{theorem}

\begin{proof}~

 \noindent $``\Rightarrow"$: This follows as a consequence of Lemmas~\ref{lem:containment} and~\ref{lem:greedypairing}.
 
 \noindent $``\Leftarrow"$: 
 We can induct on the size of $U\setminus U'$. We take a first good pair and get a quotient $N$ of $M$ by  Proposition~\ref{prop:goodpair}. Now, since we had a greedy pairing by Lemma~\ref{lem:removing} all previously good pairs remain good. Moreover, the pairing remains greedy. So we can apply induction and get $M'\leq_Q N$. By transitivity of the quotient relation we get $M'\leq_Q M$.
\end{proof}

\begin{remark}
    Note that there is a diagrammatic characterization of connected flats of LPMs in~\cite[Theorem 3.11]{Bon-06}. Since matroid quotients are well understood at the level of flats~\cite{Bry-86} this might yield another description of LPM quotients.
\end{remark}

\subsection{The quotient poset of LPMs}

Theorem \ref{thm:lpm_quotients} allows us to construct a graded poset $\mathcal P_n$ whose elements are LPMs on $[n]$ and whose ordering relation is $\leq_Q$. The left side of Figure~\ref{fig:fullP3} displays $\mathcal P_3$. It is worth mentioning that the set of matroids $\mathcal M_n$ over the set $[n]$ is endowed with a graded poset structure using the order $\leq_Q$ (see~\cite[Prop. 8.2.5]{Kun86}). However, this construction does not guarantee that the matroids obtained as quotients of a given one remain LPMs. Thus, the properties of the poset $\mathcal P_n$ that we analyze now are not obtained for free. 

\begin{proposition}
The poset $\mathcal P_n$ is graded with minimum $U_{0,n}$ and maximum $U_{n,n}$.  
\end{proposition}
\begin{proof}
 Let $M'\leq_Q M$ and consider a chain $C=(M'=M_0\leq_Q \ldots \leq_Q M_z=M)$. If two consecutive elements $M_i=M[U_i,L_i]\leq_Q M_{i+1}=M[U_{i+1},L_{i+1}]$ have non-consecutive ranks, i.e., $r(M_{i+1})-r(M_i)>1$, then by Theorem~\ref{thm:lpm_quotients}, the greedy pairing given by $M_i$ and $M_{i+1}$ on $L_{i+1}\setminus L_i,U_{i+1}\setminus U_i)$ allows us to enlarge the chain $C$ by performing quotients pair by pair. Hence, each maximal chain in the interval $[M, M']_Q$ in $\mathcal P_n$ has length $r(M')-r(M)=|U'\setminus U|$. The statement about maximum and minimum is clear, since every matroid on $n$ elements is a quotient of $U_{n,n}$ and has $U_{0,n}$ as a quotient and both are uniform hence LPMs.
\end{proof}


The curious reader might wonder whether $\mathcal P_n$ is a lattice. This, however is not the case. For instance in $\mathcal P_3$ the matroids $M[12,23]$ and $M[13,23]$ are both coverings of the matroids $M[1,3]$ and $M[1,2]$, i.e., they do not have a unique meet (see Figure~\ref{fig:fullP3}). Since the four matroids in the previous example are on two consecutive ranks and $\mathcal P_3$ is a graded subposet of the graded poset $\mathcal M_3$, this also implies that $\mathcal M_3$ is not a lattice, which was probably known before. Since $\mathcal P_3$ and $\mathcal M_3$ are induced subposets on consecutive ranks of $\mathcal P_n$ and $\mathcal M_n$, respectively, $\mathcal P_n$ and $\mathcal M_n$ are not lattices for any $n\geq 3$. 

We also point out that the poset $\mathcal P_3$ considered here is a subposet from the one considered in \cite[Section 3]{BCT} where all positroids on $[3]$, not only LPMs, are considered.

\begin{figure}[htp]
    \centering
    \includegraphics[width=.9\textwidth]{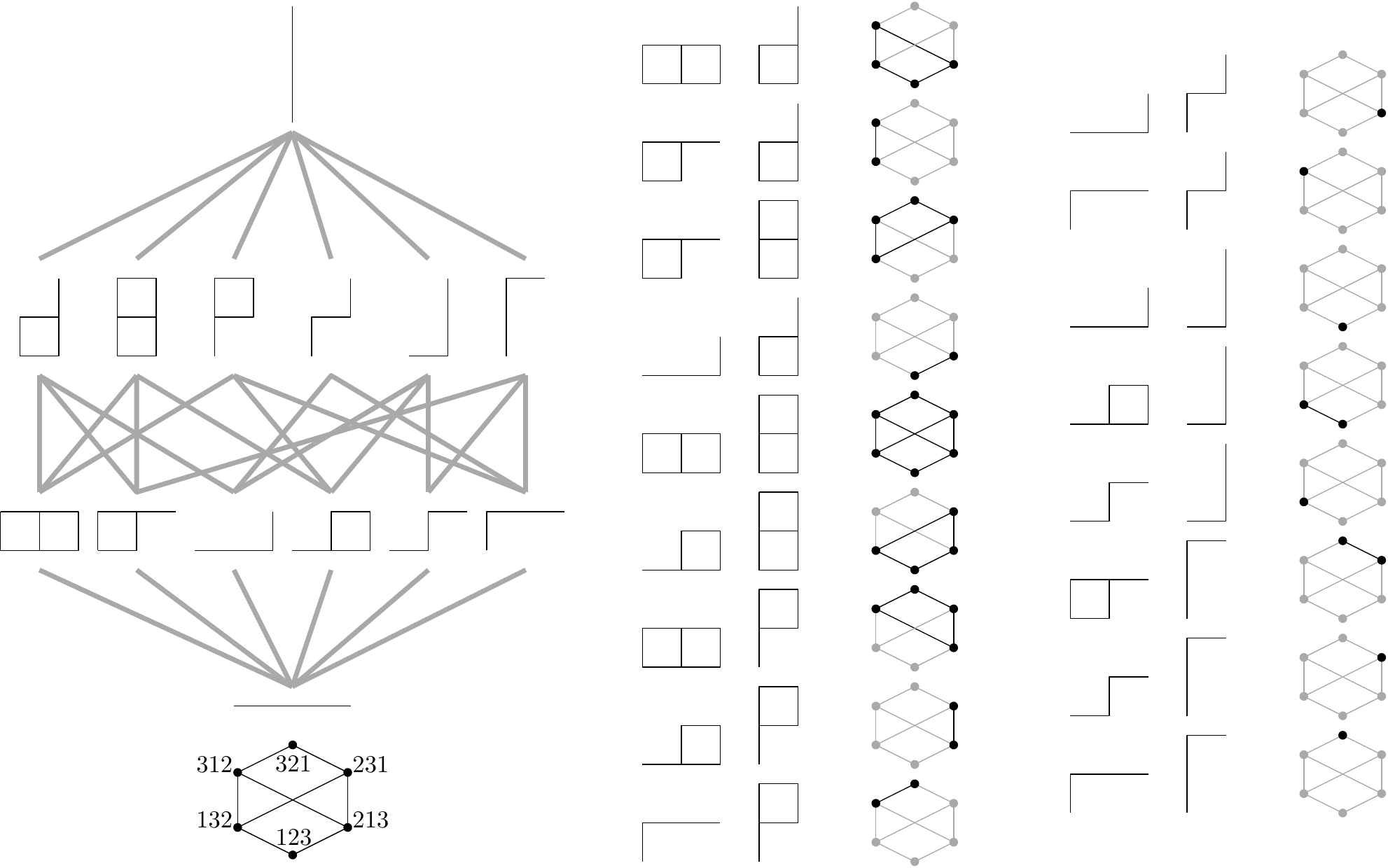}
    \caption{On the upper left the poset $\mathcal P_3$. On the lower left the strong Bruhat order $(S_3,\leq_B)$. On the right the corresponding intervals in $(S_3,\leq_B)$ given by each maximal chain, as explained in Section \ref{sec:Bruhat}.}\label{fig:fullP3}
\end{figure}


Let us explore a bit more the poset $\mathcal P_n$. For $k\in\{1,\dots,n\}$ denote by $a(n,k):=\frac{1}{n}{n \choose k}{n \choose {k-1}}$. The numbers $a(n,k)$ are known as \emph{Narayana numbers}, and count the number of Dyck paths from $0$ to $2n$ with $k$ peaks (see \cite[Exercise 6.36]{EC2} and the right of Figure~\ref{fig:dyck}). 

\begin{corollary}\label{cor:unimodular}
 The poset $\mathcal P_n$ has $a(n+1,z+1)$ elements of rank $k=n-z$, for  each  $z\in\{0,\dots,n\}$.
\end{corollary}
\begin{proof}
Our proof will be based on two observations:
\begin{itemize}
\item[(a)] By Theorem~\ref{thm:lpm_quotients}, every LPM $M=M[U,L]$ of rank $k=n-z$ corresponds to a greedy pairing $((\ell'_{j_1},u'_{i_1}),\ldots, (\ell'_{j_z},u'_{i_z}))$ on  $([n]\setminus L, [n]\setminus U)$ of length $k$ obtained from $M'=M[\{1,\ldots, n\},\{1,\ldots, n\}]= U_{n,n}$.
\item[(b)] There is a bijection between such greedy pairings $((\ell'_{j_1},u'_{i_1}),\ldots, (\ell'_{j_z},u'_{i_z}))$ and the Dyck paths from $0$ to $2(n+1)$ with $z+1$ peaks.
\end{itemize}

For part (a), if $M=M[U,L]\in\mathcal P_n$  then $M\leq_Q M'$ by Theorem~\ref{thm:lpm_quotients}, $M$ corresponds to the greedy pairing on $([n]\setminus L,[n]\setminus U)$.

For part (b) given a greedy pairing $((\ell'_{j_1},u'_{i_1}),\ldots, (\ell'_{j_k},u'_{i_k}))$, consider the sequence of points $({j_1},{i_1}),\ldots, ({j_z},{i_z})\in [n]\times[n]$. Since this is a greedy pairing we have $j_1<\cdots<j_z$, $i_1<\cdots<i_z$ and $i_r\geq j_r$ for all $1\leq r\leq z$ since all pairs are good. This is, the points sit weakly above the skew diagonal in the grid $[n]\times[n]$ and the upper left quadrant of each point is empty. Note that the properties $i_r\geq j_r$ characterizes all good pairs since we are in $U_{n,n}$. Now, adding points $(0,0)$ and $(n+1,n+1)$ allows to associate $M$ with a Dyck path from $0$ to $2(n+1)$ with $z+1$ peaks. See Figure~\ref{fig:dyck}.

%
%

\end{proof}

\begin{figure}[htp]
    \centering
    \includegraphics[width=.8\textwidth]{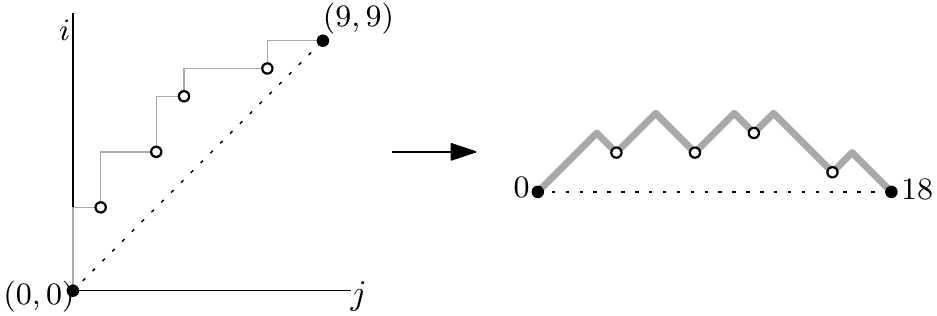}
    \caption{The LPM $M[1246,3568]$ as quotient of $U_{8,8}$ with greedy pairing $(1,3),(3,5),(4,7),(7,8)$ and the corresponding Dyck path.}\label{fig:dyck}
\end{figure}


\begin{remark}
The proof of Corollary~\ref{cor:unimodular} provides an idea of how to analyze the ranks of general intervals in the poset $\mathcal P_n$. However, this Corollary could also be argued as follows.
In order to see that the number of lattice path matroids on $[n]$ having rank $n-k$
is $a(n+1,k+1)$ follows by the fact that the number of pairs of non-crossing lattice paths
from $(0, 0)$ to $(k, n-k)$ with steps $+(1, 0)$ and $+(0, 1)$ can be calculated as a
determinant of a $2 \times 2$ matrix of binomial coefficients using the Lindström–
Gessel–Viennot lemma (with a small tweak to count those lattice path matroids
that have loops/coloops, i.e., those for which the two paths do intersect with
“overlaps”). See~\cite{Kra15}.    
\end{remark}

%

\section{LPMs and the nonnegative flag variety}\label{sec:Bruhat}

In this section we will study maximal chains in the interval $[U_{0,n},U_{n,n}]_Q$ of the poset $\mathcal P_n$. That is, we study {(full) lattice path flag matroids, LPFMs}. Recall that following Definition~\ref{def:flag_matroid} an LPFM is a sequence $\mathcal F:(M_0, M_1, \ldots, M_n)$ of LPMs where $M_0\leq_Q M_1\leq_Q\cdots\leq_Q M_n$ is a maximal chain in $\mathcal P_n$. That is, each $M_i$ is an LPM on $[n]$ and for $i=0,\dots,n-1$:
\begin{itemize}
\item[(a)] $M_i$ is a quotient of $M_{i+1}$
\item[(b)] $r(M_i)+1=r(M_{i+1})$. 
 \end{itemize}

One of the main results of our paper will show us that the family of LPFMs is included in $\mathcal F\ell_n^{\geq 0}$. That is, every  LPFM can be represented by a point in $\mathcal F\ell_n^{\geq 0}$ and thus we can think of the family of LPFMs as properly contained inside $\mathcal F\ell_n^{\geq 0}$. In order to achieve this, we will make use of \emph{matroid polytopes}, defined next. 

\begin{definition}\label{def:polytopes}Let $\{e_1,\dots,e_n \}$ be the canonical basis of $\mathbb R^n$.
\begin{itemize}
\item[(1)] Let $M$ be a matroid on $[n]$ of rank $k$ and let $\mathcal B$ its set of bases. The \emph{matroid polytope of $M$} is the polytope $\Delta_M$ in $\mathbb R^n$ given as the convex hull 
$\Delta_M:=conv\{ e_B\, |\, B\in\mathcal B\}$ where $e_B=\sum_{i\in B}e_i$.

\item[(2)] Let $\mathcal F:(M_0,\dots,M_r)$ be a flag matroid whose constituents $M_i$ are matroids on $[n]$. The \emph{flag matroid polytope} $\Delta_{\mathcal F}$ is  the polytope in $\mathbb R^n$ given by $$\Delta_{\mathcal F}:=conv\{e_{B_0}+\cdots+e_{B_r}\mid \mathcal{B}=(B_0,B_1,\dots,B_r)\text{ is a flag of bases of }\mathcal F\}.$$
\end{itemize}
\end{definition}

For those familiar with polytopes, if $\Delta_i$ denotes the matroid polytope of $M_i$ for each $M_i$ as in (2) of Definition \ref{def:polytopes}, then the  polytope $\Delta_{\mathcal F}$ is the Minkowski sum $\Delta_1+\cdots+\Delta_n$ (see \cite[Cor. 1.13.5]{BGW}). Also, notice that Definition \ref{def:polytopes}(2) does not assume the flag is full, as $r\leq n$. When $r=n$ then $\Delta_{\mathcal F}$ is such that each of its vertices is a permutation of the point $(1,2\dots,n)$. In particular if $\mathcal F$ is the \emph{uniform flag matroid} $\mathfrak{U}_n=(U_{0,n},U_{1,n},\dots,U_{n,n})$ then $\Delta_{\mathcal F}$ has $n!$ vertices, given by all the permutations of $(1,2\dots,n)$. That is, the polytope $\Delta_{\mathcal F}$ is the permutahedron. Now, notice that since $U_{n,n}$ has only one basis $B=\{12\dots n \}$ then $e_B=(1,1,\dots,1)$. Thus, 
any full flag  matroid $\mathcal F=(M_0,M_1,\dots,M_{n-1},U_{n,n})$ is such that its polytope $\Delta_{\mathcal F}$ is a translation of the polytope $\Delta_{\mathcal F'}$, by $(1,\dots,1)$, where $\mathcal F'=(M_0,M_1,\dots,M_{n-1})$, and the latter polytope has vertices which are permutations of $(0,1,\dots,n-1)$. 

\begin{example}\label{ex:pol_in_3}
Consider the LPFM given by $\mathcal F:M_0\leq_Q M_1\leq_Q M_2\leq_Q M_3$ where $M_1=U_{1,3}$, $M_2=M[13,23]$ and $M_3=U_{3,3}$. Then the flags of bases of $\mathcal F$ are
\begin{align*}
1\subset 13\subset123 \qquad & 2\subset 23\subset123\\
3\subset 13\subset123 \qquad & 3\subset 23\subset123.
\end{align*}
Each of these flags gives rise, respectively, to the points $(3,1,2),(1,3,2),(2,1,3),(1,2,3)$ in $\mathbb R^3$. Thus the polytope $\Delta_{\mathcal F}$ is the convex hull of these four points and it is depicted in Figure~\ref{fig:fullP3} along with all the polytopes arising from full flags of LPMs over the set $[3]$.
\end{example}

\begin{definition}\label{def:strong_bruhat}
Let $u,v\in S_n$. We say that \emph{$v$ covers $u$} in the (strong) \emph{Bruhat order}, denoted $u\prec_B v$ if $v=u(i,j)$ for some transposition $(i,j)$ with $i<j$ such that if $i<k<j$ then $u(k)<u(i)$ or $u(k)> u(j)$. The Bruhat order of $S_n$ is the transitive closure of this covering relation.
\end{definition} 

The next main result in this paper shows that every flag matroid polytope $\Delta_\mathcal F$  over $[n]$, where $\mathcal F:(M_0, M_1, \ldots, M_n)$ is an LPFM, is such that (its $1$-skeleton) is an interval in the strong {Bruhat order} $\leq_B$ of $S_n$. The importance of this result is that, every interval in the Bruhat order can be thought of as the 1-skeleton of a flag matroid that arises as a point of $\mathcal F\ell_n^{\geq 0}$. In Example \ref{ex:pol_in_3} the 1-skeleton of $\Delta_{\mathcal F}$ corresponds to the interval $[123,312]_B$ in $S_3$. 
Conversely, as shown in \cite[Proposition 2.7]{TW15} and~\cite[Theorem 6.10]{KW15}, every flag matroid $\mathcal F$ arising from a point in $\mathcal F\ell_n^{\geq 0}$ is such that its flag matroid polytope is (its 1-skeleton) an interval in the (strong) Bruhat order $S_n$. This correspondence is found in terms of moment maps in the flag variety as follows.
\begin{theorem}\cite[Theorem 6.10]{KW15}
    Let $g\in \mathcal F\ell_{v,w}^{>0}$. Then its polytope image under the moment map is the polytope $P_{v,w}$ whose vertices are $\{z:v\leq_B z\leq_B w\}$.
\end{theorem}
    
Polytopes of the form $P_{v,w}$ are referred to as \emph{Bruhat interval polytopes} in \cite{TW15}.

Let $\mathcal{F}:(M_0, M_1, \ldots, M_n)$ be an LPFM. Given two flags of bases of $\mathcal F$, namely $\mathfrak B:(B_0, B_1, \ldots, B_n)$ and $\mathfrak B':(B'_0, B'_1, \ldots, B'_n)$, we say that $\mathfrak B$ is \emph{smaller} than $\mathfrak B'$ if and only if $B_i\leq_G B'_i$ for all $i\in[n]$. We denote this as  $\mathfrak B\leq_G\mathfrak B'$. We say that the permutation $\pi=\pi_{\mathfrak B}$ associated to the flag of bases $\mathfrak B$ is the permutation in $S_n$ such that $\pi(i)=B_i\setminus B_{i-1}$, for $i=1,\dots,n$. We refer to $\pi$
 as the \emph{Gale permutation of $\mathfrak B$}. On the other hand, the \emph{Bruhat permutation of $\mathfrak B$} is the permutation $\tau=\tau_{\mathfrak B}$ in $S_n$ such that $\tau(i)=\overline\pi^{-1}(i)$ where $\overline\pi(i)=\pi(n-i+1)$. It is worth pointing out that such $\mathfrak B$ determines $\pi$ (and thus $\tau$) uniquely. Thus  we will say that  $\pi=\pi_{\mathfrak{B}}\leq_G\pi_{\mathfrak{B}'}=\pi'$ if and only if $\mathfrak B\leq_G\mathfrak B'$, where  $\mathfrak B$ and  $\mathfrak B'$ are flags of bases of the uniform flag matroid $\mathfrak U_n=(U_{0,n},U_{1,n},\dots,U_{n,n})$.




 \begin{example}
 Consider the LPFM $\mathcal F:M_0\leq M_1\leq_Q M_2\leq_Q M_3\leq_Q U_{4,4}$ where $M_1=M[1,3]$, $M_2=M[14,34]$ and $M_1=M[124,134]$. The polytope $\Delta_{\mathcal F}$ is the convex hull of 6 points in $\mathbb R^4$. Each point arises from each flag of bases of $\mathcal F$, which the reader can compute.\footnote{We have omitted $U_{4,4}$ as it provides no further information.}
In Figure \ref{fig:exP4} we depict on the left the constituents of $\mathcal F$. Below each of them appear their bases set. Also, each covering relation $\prec_q$ is labelled by the corresponding good pair.
On the right hand side appears the interval $[1243,4213]_B$ in $S_4$ whose permutations correspond, bijectively, to the vertices of $\Delta_{\mathcal F}$, i.e. to the collection of flags of bases of $\mathcal F$. For instance the flag of bases $\mathfrak B:(3,34,134,1234)$ is such that its Bruhat permutation $\tau=2143$ corresponds to the vertex $(2,1,4,3)$. On the other hand, its Gale permutation is  $\pi=3412$. If $\mathfrak B':(3,34,234,1234)$ then its corresponding Gale and Bruhat permutations are $\pi'=3421$, $\tau'=1243$. Moreover,  $\pi'\geq_G\pi$ and $\tau'\leq_B\tau$.

\begin{figure}[htp]
    \centering
    \includegraphics[width=.9\textwidth]{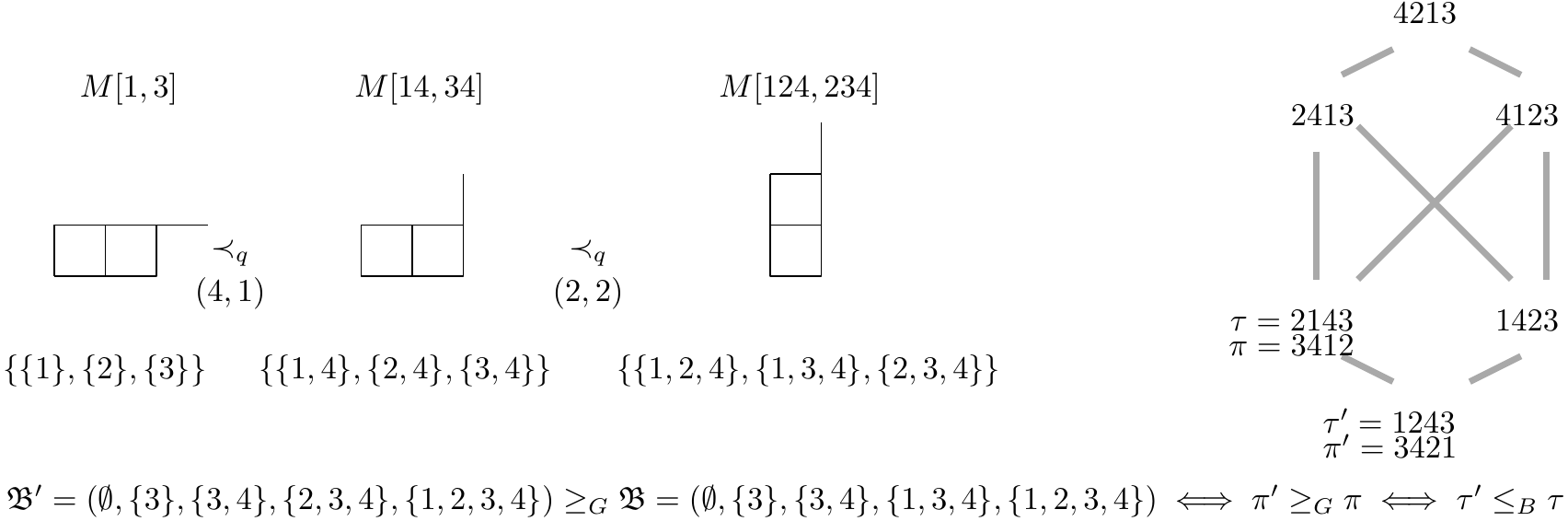}
    \caption{An LPFM and its flag matroid polytope. Its vertices constitute the interval $[1243,4213]_B$ in the Bruhat order.}\label{fig:exP4}
\end{figure}
\end{example}

\begin{lemma}\label{lemma:bruhat_lpm}
Let $\tau', \tau\in S_n$ the Bruhat permutations of flags $\mathfrak B', \mathfrak B$, respectively. If $\tau'\prec_B\tau$ then $\mathfrak B'\geq_G \mathfrak B$.\end{lemma}
\begin{proof} Let $\tau'=i_1\cdots i_n$.
Now, $\tau'\prec_B\tau$ if and only if $\tau=\tau'(r,s)$ for some $r<s$ such that $i_r<i_s$ and if $r<t<s$ then $i_t\notin[i_r,i_s]$.  Thus $\tau$ is obtained from $\tau'$ by exchanging positions $r$ and $s$. In view of this, we get that $B_j$, the $j$-th component of $\mathfrak B$, satisfies
$$\begin{cases}
B_{j}=B_{j}'\setminus\{s\}\cup\{r\} & \text{ if }j\in\{n-i_s+1,\dots,n-i_r\}\\
B_{j}=B_{j}'& \text{ otherwise}
\end{cases}.$$
The result follows. 
\end{proof}


\begin{lemma}\label{lem:coveringG}
Let $\mathfrak B, \mathfrak B'$ flags of bases of $\mathfrak{U}_n$ and $\pi:=\pi_{\mathfrak B}\prec_G\pi_{\mathfrak B'}=:\pi'$. Then there are $i<j\in[n]$ such that $\pi(k)=\pi'(k)$ for all $k\in[n]\setminus\{i,j\}$, $\pi(i)<\pi'(i)$, $\pi(i,j)=\pi'$ and $\pi'(l)\notin[\pi(i),\pi'(i)]$ for all $i<l<j$.
\end{lemma}
\begin{proof} Let $\pi=a_1\cdots a_n\in S_n$ and let $\pi'\neq \pi$ and denote by $B_{r,\pi}$ the $r$th set of $\mathfrak B$, and similar for $\pi'$. That is, there is at least an index $i$ such that $a_i\neq \pi'(i)$. Notice then that there must be another index $j\neq i$ with the same property and we may assume $i<j$. Now we prove the contrapositive. That is we assume that if $\pi,\pi'$ differ in more than 2 positions, or, if there exists $\ell$ such that $i<\ell<j$ and $a_\ell\in[a_i,a_j]$, then $\pi'$ does not cover $\pi$ in the Gale order. 
Suppose that $\pi'$ and $\pi$ differ in at least 3 places. Pick the first 3 positions where they differ, say $i<k<j$. Thus in one-line notation the first $j$ values in $\pi'$ are $b_1\cdots b_i\cdots b_k\cdots b_j$ where $b_r=a_r$ for all $r\in[j]\setminus\{i,k,j \}$. Moreover, $b_ib_kb_j=a_ka_ja_i$ or $b_ib_kb_j=a_ja_ia_k$, otherwise the choice of $i,k,j$ is contradicted. If $a_i<a_k<a_j$ then $\pi<_G\pi'$ but is not a covering since, for instance, $\pi<_G\pi(i,k)<_G\pi'$. This can be seen by noticing that $B_{r,\pi}\leq_G B_{r,\pi'}$ for all $r$, by the relative order of $a_i<a_k<a_j$. If $a_i<a_k>a_j$ then $\pi,\pi'$ are not comparable if $a_j<a_i$, since $B_{r,\pi}\leq_G B_{r,\pi'}$ for $r<i$ but $B_{i,\pi}>_G B_{i,\pi'}$. A similar analysis holds for the remaining cases that compare the relative order of the triple $a_ia_ka_j$, leading us to either of the two conclusions displayed here. Hence, if $\pi,\pi'$ differ in more than 2 positions then $\pi'$ does not cover $\pi$. 
Now, we assume that there exists $\ell$ such that $i<\ell<j$, $a_\ell\in[a_i,a_j]$, and $\pi(i,j)=\pi'$. One checks that for all $r\in [n]$ $B_{r,\pi}\leq_G B_{r,\tau}\leq_G B_{r,\pi'}$ where $\tau=\pi(i,l)$. Therefore $\pi'$ does not cover $\pi$. The claim follows. 
 \end{proof}

\begin{lemma}\label{lemma:lpm_bruhat}
Let $\mathfrak B, \mathfrak B'$ be flags of bases of $\mathfrak{U}_n$ and let $\pi:=\pi_{\mathfrak B},\pi':=\pi_{\mathfrak B'}$ their respective Gale permutations, and $ \tau:=\tau_{\mathfrak B},\tau':=\tau_{\mathfrak B'}$ their corresponding Bruhat permutations. Suppose that there are $i<j\in[n]$ such that $\pi(k)=\pi'(k)$ for all $k\in[n]\setminus\{i,j\}$, $\pi(i)<\pi'(i)$, $\pi(i,j)=\pi'$ and $\pi'(k)\notin[\pi(i),\pi'(i)]$ for all $i<k<j$. Then $\tau'\prec_B\tau$.
\end{lemma}
\begin{proof} Set $r=i-1$ and $s=j-1$.
By assumption we can write $\pi=a_1\cdots a_r a_i b_1\cdots b_s a_j c_1\cdots c_t$ and $\pi'=a_1\cdots a_r a_j b_1\cdots b_s a_i c_1\cdots c_t$, in one-line notation, where $b_{\ell}\notin[a_i,a_j]$. Therefore $\tau$ and $\tau'$ coincide for every $k\in[n]\setminus\{a_i,a_j\}$. It holds that $\tau(a_i)=n-r$, $\tau(a_j)=n-(r+s)$ and $\tau'(a_i,a_j)=\tau$. Thus we only need to show that $\tau(k)\notin[n-r, n-(r+s)]$ for $a_i<k<a_j$.
There are two cases to consider. 

\emph{Case 1:} $b_{\ell}<i$. The values belonging to the interval $[n-r, n-(r+s)]$ in $\tau$ correspond precisely to the positions $b_{\ell}$, as $\tau$ records the order of appearance of each element from $\pi$. Hence, the values in $[n-r, n-(r+s)]$ are assigned to positions to the left of $a_i$ in $\tau$. We conclude that $\tau'\prec_B\tau$.

\emph{Case 2:}  $b_{\ell}>j$. This is analogous to Case 1. In this situation  the values in $[n-r, n-(r+s)]$ are assigned to positions to the right of $a_j$ in $\tau$. The result is proven.
\end{proof}

The following result asserts that there is an order-reversing (or antitone) map between the Bruhat order $(S_n,\leq_B)$ and the Gale order $(\mathfrak U_n,\leq_G)$.

\begin{theorem}\label{thm:iso_posets}
Let $\mathfrak B,\mathfrak B'$ be flags of bases of $\mathfrak{U}_n$ and $\pi_{\mathfrak B},\pi_{\mathfrak B'}$ and $\tau_{\mathfrak B},\tau_{\mathfrak B'}$ their Gale and Bruhat  permutations as above. The following are equivalent:
\begin{itemize}
 \item[(i)]  $\mathfrak B\leq _G \mathfrak B'$,
 \item[(ii)] $\pi_{\mathfrak B}\leq _G\pi_{\mathfrak B'}$,
 \item[(iii)] $\tau_{\mathfrak B}\geq_B\tau_{\mathfrak B'}$.
\end{itemize}
\end{theorem}
\begin{proof}
 The equivalence of $(i)$ and $(ii)$ is just by definition. Lemma~\ref{lemma:bruhat_lpm} shows $(iii)\implies (i)$. Finally, $(ii)\implies (iii)$ follows by first applying Lemma~\ref{lem:coveringG} and then Lemma~\ref{lemma:lpm_bruhat}.
\end{proof}

In Figure \ref{fig:fullP3} we see that all but 2 intervals in the Bruhat order $S_3$ come from an LPFM. The ones that do not arise this way are $[132,231]$ and $[213,312]$. The former gives rise to $2\subset 23\subset 123\geq_G 2\subset 12\subset 123$ which is not an LPFM since $(3,1)$ is not a good pair for the matroid $U_{2,3}$. The reader can verify that the latter is neither an LPFM.

\begin{corollary}\label{cor:lpmint_bruhatint} 
Every lattice path flag matroid polytope is a Bruhat interval polytope.
\end{corollary}
\begin{proof}

Let $\mathcal F:(M_0, M_1, \ldots, M_n)$ be an LPFM, with $M_i=M[U_i,L_i]$ for all $0\leq i\leq n$. By Theorem~\ref{thm:iso_posets} we can argue directly in the order $\leq_G$ on the flags. We show that the set of flags of bases $\mathcal{F}$ coincides with the interval $[(U_0,\ldots,U_n),(L_0,\ldots,L_n)]_G$. The inclusion $``\subseteq''$, follows since by definition every flag $(B_0,B_1,\dots,B_n)$ of bases of $\mathcal{F}$ must be such that $U_i\leq_G B_i\leq_G L_i$, for all $i=0,1,\dots,n $.

To see the reverse inclusion $``\supseteq''$, let $\mathfrak{B}=(B_0,\ldots, B_n)\in[(U_0,\ldots,U_n),(L_0,\ldots,L_n)]_G$. Thus, $B_i\in [U_i,L_i]_G$ for all $0\leq i\leq n$. Now, by Observation~\ref{obs:weak} this simply means that $B_i$ is a base of $M_i=M[U_i,L_i]$. Hence, $\mathfrak{B}$ is a flag of bases of $\mathcal{F}$.
\end{proof}

In~\cite{BEW22} one of the results claims that it is possible to characterize when a flag matroid
polytope comes from a Bruhat interval, by just checking a condition on all the
2-dimensional faces. In light of Corollary \ref{cor:lpmint_bruhatint} it would be interesting what these faces should look like for a LPFM.
We can rephrase Corollary \ref{cor:lpmint_bruhatint} by saying that if $\mathcal F$ is an LPFM then its matroid polytope $\Delta_{\mathcal F}$ is such that its $1$-skeleton corresponds to an interval in $(S_n,\leq_B)$.
It is however not that easy to decide which intervals arise from LPFMs. 

The following Theorem establishes in terms of Gale permutations and Bruhat permutations, the condition for a sequence of LPMs to be a flag matroid. To this end we will make use of Definition \ref{def:good_pair} and translate it in terms of the aforementioned permutations. We will make use of the \emph{standardization map} $\text{st}_S:S\rightarrow[\ell]$ where $S$ is a $\ell$-subset of positive integers. The map $\text{st}_S$ is the unique bijection from $S$ to $[\ell]$ that preserves order. We also denote by $\pi([k])=\{\pi(1),\dots,\pi(k)\}$ whenever $\pi\in S_n$ and $1\leq k\leq n$.

\begin{theorem}\label{thm:quotients_asgoodpairs}
Let $\mathfrak B\leq_G\mathfrak B'$ be flags of $\mathfrak{U}_n$ and $\pi_{\mathfrak B}\leq_G\pi_{\mathfrak B'}$ and $\tau_{\mathfrak B}\geq_B\tau_{\mathfrak B'}$ the permutations associated as above. The following are equivalent:
\begin{itemize}
 \item[(i)]  the order-interval $[\mathfrak B,\mathfrak B']_G$ constitutes the set of flags of bases of an LPFM,
 \item[(ii)] for all $1\leq k\leq n$ the maps $\text{st}_{\pi_{\mathfrak{B}}([k])}:\pi_{\mathfrak{B}}([k])\rightarrow[k]$ and $\text{st}_{\pi_{\mathfrak{B}'}([k])}:\pi_{\mathfrak{B}'}([k])\rightarrow[k]$ are such that $\max\{0, \pi_ {\mathfrak B}(k)-\pi_{\mathfrak B'}(k)\}\leq \text{st}_{\pi_{\mathfrak{B}([k])}}(\pi_{\mathfrak{B}}(k))- \text{st}_{\pi_{\mathfrak{B}'}([k])}(\pi_{\mathfrak{B}'}(k))$,

 \item[(iii)] for every $1\leq k\leq n$ let $a_k=\tau_{\mathfrak B}^{-1}(n-k+1)$, $a'_k=\tau_{\mathfrak B'}^{-1}(n-k+1)$. Then $\max\{0, a_k-a_k'\}\leq \text{st}_{\{a_1,\ldots, a_k\}}(a_k)-\text{st}_{\{a'_1,\ldots, a'_k\}}(a'_k)$.
 \end{itemize}
\end{theorem}
\begin{proof} 
 ``(i)$\iff$ (ii)'': 
 This equivalence boils down to translating Definition~\ref{def:good_pair} in terms of the Gale permutations $\pi_{\mathfrak B'}$ and  $\pi_{\mathfrak B'}$. Let $\mathfrak B=(B_0,\ldots, B_n)$ and $\mathfrak B'=(B'_0,\ldots, B'_n)$. Let 
 $M_k:=M[B_k,B'_k]$ for $1\leq k\leq n$. Then by Theorem~\ref{thm:lpm_quotients} we have that
$M_{k-1}\leq_QM_k$ if and only if $(\pi_{\mathfrak B'}(k),\pi_{\mathfrak B}(k))$ is a good pair of $M_k$. Now, the map $\text{st}_{\pi_{\mathfrak{B}}([k])}$ tells us the ordering of the elements in the set $\pi_{\mathfrak{B}}([k])$, and similarly $\text{st}_{\pi_{\mathfrak{B'}}([k])}$. Thus using Definition~\ref{def:good_pair} we have that $(\pi_{\mathfrak B'}(k),\pi_{\mathfrak B}(k))$ is a good pair of $M_k$
 if and only if $0\leq\text{st}_{\pi_{\mathfrak{B}([k])}}(\pi_{\mathfrak{B}}(k))- \text{st}_{\pi_{\mathfrak{B}'}([k])}(\pi_{\mathfrak{B}'}(k))$ and $\pi_ {\mathfrak B}(k)-\pi_{\mathfrak B'}(k)\leq \text{st}_{\pi_{\mathfrak{B},k}}(\pi_{\mathfrak{B}}(k))- \text{st}_{\pi_{\mathfrak{B}'}([k])}(\pi_{\mathfrak{B}'}(k))$ which in turn is equivalent to $\max\{0, \pi_ {\mathfrak B}(k)-\pi_{\mathfrak B'}(k)\}\leq \text{st}_{\pi_{\mathfrak{B},k}}(\pi_{\mathfrak{B}}(k))- \text{st}_{\pi_{\mathfrak{B}'}([k])}(\pi_{\mathfrak{B}'}(k))$.


 \noindent ``(ii)$\iff$ (iii)'': 
 For this equivalence we recall that $\pi_B(k)=\tau_{\mathcal B}^{-1}(n-k+1)=a_k $, and similarly for $\mathcal B'$, for $1\leq k\leq n$. Thus $\max\{0, \pi_ {\mathfrak B}(k)-\pi_{\mathfrak B'}(k)\}\leq \text{st}_{\pi_{\mathfrak{B}([k])}}(\pi_{\mathfrak{B}}(k))- \text{st}_{\pi_{\mathfrak{B}'}([k])}(\pi_{\mathfrak{B}'}(k))$ if and only if $\max\{0, a_k-a_k'\}\leq \text{st}_{\{a_1,\ldots, a_k\}}(a_k)-\text{st}_{\{a'_1,\ldots, a'_k\}}(a'_k)$. The result follows.
\end{proof}

\begin{example} We illustrate Theorem \ref{thm:quotients_asgoodpairs} with two flags  $\mathfrak B$ and $\mathfrak B'$ in $\mathfrak{U}_4$ whose Gale permutations are, respectively, $\pi_{\mathfrak B}=2413$, $\pi_{\mathfrak B'}=4321$. Thus, the Bruhat permutations are, respectively, $\tau_{\mathfrak B}=2413$ and $\tau_{\mathfrak B'}=1234$. Notice that $\pi_{\mathfrak B}\leq_G\pi_{\mathfrak B'}$ and $\tau_{\mathfrak B}\geq_B\tau_{\mathfrak B'}$. Following the notation in the proof of the theorem, setting $k=3$ we summarize as follows the calculations needed to verify the condition to be a good pair. 
However notice that in order to verify $\mathfrak B\leq_G\mathfrak B'$ one needs to do the corresponding calculations for every $k\in[n]$.

\begin{table}[h]
{\footnotesize
\begin{tabular}{c|c|c|c|c|c|c}

 $\pi_{\mathfrak B}([3])$ &  $\pi_{\mathfrak B}(3)$ & $\text{st}_{\pi_{\mathfrak{B}}([3])}(\pi_{\mathfrak{B}}(3))$ &   $\pi_{\mathfrak B'}([3])$ &  $\pi_{\mathfrak B'}(3)$ & $\text{st}_{\pi_{\mathfrak{B'}}([3])}(\pi_{\mathfrak{B}}(3))$ & $u_j-\ell_i\leq j-i$ \\
 \hline
 $\{1,2,4\}$ & 1 & 1 &  $\{2,3,4\}$ & 2 & 1 & $1-2\leq 1-1$\\
 \hline
 \hline
  $\tau^{-1}_{\mathfrak B}(\{4,3,2\})$ &  $\tau^{-1}_{\mathfrak B}(2)$ & $\text{st}_{\tau^{-1}_{\mathfrak{B}}([2])}(\tau^{-1}_{\mathfrak{B}}(3))$ &   $\tau^{-1}_{\mathfrak B'}([\{4,3,2\})$ &  $\tau^{-1}_{\mathfrak B'}(2)$ & $\text{st}_{\tau^{-1}_{\mathfrak{B'}}([3])}(\tau^{-1}_{\mathfrak{B}}(2))$ & $u_j-\ell_i\leq j-i$ \\
 \hline
 $\{1,2,4\}$ & 1 & 1 &  $\{2,3,4\}$ & 2 & 1 & $1-2\leq 1-1$
\end{tabular}}
\end{table}

\end{example}

Our next result establishes that some particular intervals in the Bruhat order come from LPFMs. 

\begin{proposition}\label{thm:cubes_are_lpm}
Let $s_i=(i,i+1)$ be a simple transposition in $S_n$. Let $\tau,\tau'$ be permutations of $S_n$ such that $\tau\leq_B\tau'$  where $\tau'=\tau s_{i_1}\cdots s_{i_m}$ for some $i_1,\dots,i_m\in[n-1]$. If the $s_{i_j}$ commute pairwise then $[\mathfrak B', \mathfrak B]_G$ constitute the set of flags of bases of an LPFM on $[n]$.
\end{proposition}
\begin{proof}
Let $I_1=\{i_1,...,i_m\}$ and $I_2=\{i_1+1,...,i_m+1\}$. Then 
$$\tau'(i)=\begin{cases}
    \tau(i) & i\in[n]\setminus(I_1\sqcup I_2)\\
    \tau(i+1) & i\in I_1\\
    \tau(i-1) & i\in I_2\\
\end{cases}\Rightarrow\
a'_k-a_k=\begin{cases}
   0  & \tau^{-1}(n-k+1) \in[n]\setminus(I_1\sqcup I_2)\\
   1 &\tau^{-1}(n-k+1)  \in I_1\\
   -1 & \tau^{-1}(n-k+1) \in I_2\\
\end{cases}.$$
On the other hand, since $\{a_1,\dots,a_k\}=\{\tau^{-1}(n),\dots,\tau^{-1}(n-k+1)\}$ then $\text{st}_{\{a_1,\dots,a_k\}}(a_k)=\text{st}_{\{a'_1,\dots,a'_k\}}(a'_k)$ if $\{a_1,\dots,a_k\}\subseteq[n]\setminus(I_1\sqcup I_2)$, or, $i_r\in\{a_1,\dots,a_k\}$ implies $i_r+1\notin\{a_1,\dots,a_k\}$. That is, the relative position of $i_r$ in $\{a_1,\dots,a_k\}$ is the same as that of $i_{r}+1$ in $\{a'_1,\dots,a'_k\}$ as long as $i_r+1$ has not been added yet to $\{a_1,\dots,a_k\}$ (and thus $i_r$ has not been added yet to $\{a'_1,\dots,a'_k\}$).
Otherwise, $\text{st}_{\{a'_1,\ldots, a'_k\}}(a'_k)-\text{st}_{\{a_1,\ldots, a_k\}}(a_k)=1$. Summarizing we have

\begin{table}[h]
\begin{tabular}{|c|c|}
\hline
$a_k'-a_k $&$ \text{st}_{\{a'_1,\ldots, a'_k\}}(a'_k)-\text{st}_{\{a_1,\ldots, a_k\}}(a_k)$ \\
\hline
$0$&$ 0$ \\
\hline
$-1$&$ 0$ \\
\hline
$1$&$ 1 $\\
\hline
\end{tabular}
\caption{\label{tab:prop_cubes}Proof of Proposition \ref{thm:cubes_are_lpm}.}
\end{table}
The result then follows from Theorem~\ref{thm:quotients_asgoodpairs}.
\end{proof}

We close this section proving that LPM quotients are realizable.
In \cite{MOX19} the authors consider \emph{realizable quotients}, which we denote with $\trianglelefteq_Q$. Namely, if $M'$ and $M$ are positroids over $[n]$ of ranks $k<\ell$, respectively, then $M'\trianglelefteq_Q M$ if there exists a point $A\in Gr_{\ell,n}^{\geq 0}$ such that $A$ represents $M$ and the submatrix  $A'$ obtained from $A$ by keeping its top $k$ rows is such that $A'$ represents $M'$ and $A'\in Gr_{k,n}^{\geq 0}$. 

\begin{corollary}\label{cor:realizable}
    LPM quotients are realizable. That is, if $M'$ and $M$ are LPMs on $[n]$ and $M'\leq_Q M$ then $M'\trianglelefteq_Q M$.
\end{corollary}
\begin{proof}
If $M'$ and $M$ are LPMs on $[n]$ and $M'\leq_Q M$, then by Theorem~\ref{thm:lpm_quotients} there is an LPFM $\mathcal{F}:(M_0\leq_Q\ldots\leq_Q M_n)$ with $M'=M_i$ and $M=M_j$ for some $0\leq i<j\leq n$. By Corollary~\ref{cor:lpmint_bruhatint}, $\mathcal{F}$ corresponds to an interval of the (strong) Bruhat order $S_n$. Now, by~\cite[Proposition 2.7]{TW15} or~\cite[Theorem 6.10]{KW15}, $\mathcal{F}$ can be thought of as a point of $\mathcal F\ell_n^{\geq 0}$. In particular, $M'\trianglelefteq_Q M$.
\end{proof}

We have shown that given two LPMs $M'\leq_Q M$ there exists a representable flag matroid that has them as constituents. This is not true for realizable matroids in general, see~\cite[1.7.5 Example 7]{BGW}. Our results moreover show that there exits a point $A\in \mathcal F\ell_n^{\geq 0}$ that realizes simultaneously $M$ and $M'$. This is not true for positroids in general, as pointed out in Example~\ref{xmpl:postroids} and not even if they are quotients in the more restrictive setting of oriented matroids, see~\cite[Example 4.6]{BEW22}.

\section{On a conjecture of Mcalmon, Oh, and Xiang}

In this section we will prove a conjecture made by Mcalmon, Oh, and Xiang~\cite{MOX19,OX22} which aims to characterize quotients of positroids (with no loops or coloops) combinatorially in the special case of LPMs. As we already know, LPMs are a subfamily of positroids and thus, our purpose now is to state and prove this conjecture for LPMs using the results we have developed already. Recall that if $A\subseteq [n]$ then $\overline A$ denotes the set $[n]\setminus A$. 

\begin{definition}\label{def:c_r_intervals_lpm}
Let $M=M[U,L]$ be an LPM over $[n]$ where $U=\{ u_1<\cdots<u_k\}$ and $L=\{ \ell_1<\cdots<\ell_k\}$. Let $\overline{L}=\{\overline{\ell}_1< \cdots <\overline{\ell}_{n-k}\}$ and $\overline{U}=\{\overline{u}_1< \cdots <\overline{u}_{n-k}\}$ and assume that $M$ has no loops nor coloops.
\begin{itemize}
\item[(RI)] A \emph{row-interval} of $M$ is a cyclic interval of the form $\{\ell_i,\ell_i+1,\dots,n,1,\dots,u_i \}$, for every $i\in\{1,\dots,k \}$. We denote such an interval by $[\ell_i,u_i]$.
\item[(CI)] A \emph{column-interval} of $M$ is an interval of the form $\{\overline{\ell_i},\overline{\ell_i}+1,\dots,\overline{u_i} \}$, for every $i\in\{1,\dots,n-k \}$. We denote such an interval by $[\overline{\ell_i},\overline{u_i}]$. 
\end{itemize}
An \emph{interval} of $M$ is either a row or a column interval of $M$.
\end{definition}

There is a bijective correspondence between positroids on $[n]$ and \emph{decorated permutations} on $[n]$  (see \cite{P06}), i.e, bijections from $[n]$ to $[n]$, where fixed points are additionally \emph{decorated} with an underline $\pi(a)=\underline a$ or not.
Let now $M$ be an LPM as in Definition \ref{def:c_r_intervals_lpm}. The {decorated permutation $\pi_M$}, or simply $\pi$, associated to $M$ is the permutation on the set $[n]$ given by 
$$\begin{cases}
\pi(u_i)=\ell_i & \text{ for }i\in\{1,\dots,k \}\\
\pi(\overline{u_i})=\overline{\ell_i} & \text{ for }i\in\{1,\dots,n-k \}.
\end{cases}$$
If $a\in[n]$ is a loop of $M$, then $\pi(a)=\underline a$. If $a\in[n]$ is a coloop of $M$ then $\pi(a)=a$. That is, loops and coloops are the only fixed points of $\pi$ and they are either decorated with an underline or not decorated, respectively. However since we are considering $M$ to be loop-free and coloop-free, then no fixed points will arise in the corresponding permutation $\pi$.

We illustrate these concepts with an example. For our purposes the definition we are providing here for such permutations, has been adapted to LPMs. Also, sometimes in the literature the definition given for decorated permutation would differ from ours by taking the inverse  $\pi^{-1}$, of the one we provided here.

As an example, consider the LPM given by $M'=[13,25]$ over the set $[5]$. Then its decorated permutation in one-line notation is $\pi=21534$. Also, the row-intervals of $M'$ are $[2,1]=\{2,3,4,5,1 \}$ and $[5,3]=\{5,1,2,3 \}$. On the other hand, its column-intervals are $[1,2]=\{ 1,2\}$, $[3,4]= \{ 3,4\}$ and $[4,5]=\{ 4,5\}$. In general, given any $M'$ as in Definition \ref{def:c_r_intervals_lpm}, it follows that each of its row-intervals contains the set $\{1,n\}$. On the other hand, the only column-interval that contains $n$ is the right-most column-interval $[\overline{\ell_{n-k}},n]$ and the only interval that contains $1$, is the left-most column-interval $[1,\overline{u_1}]$. Thus, if a column-interval of $M$ is expressed as a union containing a row-interval of $M'$, then it has to be simultaneously the left-most and right-most column-interval of $M$. Hence $M=U_{n-1,n}$.
This discussion leads us to the following observation which will be used throughout in the proof of Theorem~\ref{thm:suho}.

\begin{observation}\label{obs:columnsbecolumns}
 Let $M',M$ be LPMs on $[n]$ that are loop and coloop free. Also assume that $M\neq U_{n-1,n}$. If a column-interval of $M$ is expressed as union of intervals of $M'$, then these are all column-intervals.\end{observation}

It is worth pointing out that Observation \ref{obs:columnsbecolumns} does not hold in general for row-intervals. For instance consider $M=[123,245]$ and $M'=[13,25]$. Then the row-interval $[4,2]=\{1,2,4,5\}$ of $M$ can only be represented as union of column-intervals $[1,2]\cup [4,5]$. Now, in \cite{MOX19}, what the authors call \emph{CCW}-arrows of an arbitrary positroid, correspond in the case of an LPM to its intervals as given in Definition \ref{def:c_r_intervals_lpm}. We can now state the following.

\begin{conjecture}[Mcalmon, Oh, Xiang '19]\label{conj:MOX}
 For positroids $M',M$ we have $M'\trianglelefteq_Q M$ if and only if every CCW-arrow of $M$ is the union of  CCW-arrows of $M'$. 
\end{conjecture}

Now we are ready to prove and strengthen this conjecture for LPMs.


\begin{theorem}\label{thm:suho}
 Let $M'$ and $M$ be LPMs on $[n]$ without loops or coloops. The following are equivalent:
 \begin{itemize}
  \item[(i)] $M'\leq_Q M$,
  \item[(ii)] $M'\trianglelefteq_Q M$,
  \item[(iii)] every interval of $M$ can be expressed as union of intervals of $M'$.
 \end{itemize}
\end{theorem}
\begin{proof}~

\noindent (i) $\Rightarrow$ (ii): This is the content of Corollary~\ref{cor:realizable}.

\noindent (ii) $\Rightarrow$ (i): This follows by definition.

 \noindent (iii) $\Rightarrow$ (i): 
 If $M=U_{n-1,n}$, then all coloop-free matroids on $[n]$ are quotients of $M$ and we are done. 
 Now let $M'=M[U',L']$, $M=M[U,L]$ with $M\neq U_{n-1,n}$ and assume that (iii) holds.  In order to prove (i) we will show that $U'\subseteq U, L'\subseteq L$ and that the greedy pairing, as given in Definition \ref{def:good_pairing}, is good.
 
Denote $[n]\setminus U=\overline{U}=\{\overline{u}_1<\cdots< \overline{u}_{n-k}\}$, $[n]\setminus L=\overline{L}=\{\overline{\ell}_1<\cdots <\overline{\ell}_{n-k'}\}$, $[n]\setminus U'=\overline{U'}=\{\overline{u'}_1<\cdots< \overline{u'}_{n-k'}\}$, and $[n]\setminus L'=\overline{L'}=\{\overline{\ell}_1'<\cdots <\overline{\ell}_{n-k'}'\}$. By hypothesis, and using Observation~\ref{obs:columnsbecolumns}, every column-interval $[\overline{\ell},\overline{u}]$ of $M$ can expressed as union of intervals $\bigcup_{s=1}^t[\overline{\ell'}_{i_s},\overline{u'}_{i_s}]$ in $M'$, where each of the intervals $[\overline{\ell'}_{i_s},\overline{u'}_{i_s}]$ is a column-interval of $M'$. In particular, $\overline{\ell'}_{i_s}=\overline{\ell}$ and $\overline{u'}_{i_t}=\overline{u}$. Hence, $\overline{\ell}\in\overline {L'}$ and $\overline{u}\in \overline {U'}$. Therefore, $\overline{U'}\supseteq \overline{U}$ and $\overline{L'}\supseteq \overline{L}$ and hence $U'\subseteq U$ and $L'\subseteq L$. Moreover, we have $\rank(M')\leq \rank(M)$ and $M=M'$ if $\rank(M)=\rank(M')$.

 
Now, letting $U\setminus U'=\{u_{i_1}<\cdots< u_{i_z}\}$ and $L\setminus L'=\{\ell_{j_1}<\cdots< \ell_{j_z}\}$, take the greedy pairing $((\ell_{j_1},u_{i_1}),\ldots, (\ell_{j_z},u_{i_z}))$. In order to prove (i), it suffices to show that the greedy pairing is good, by Theorem \ref{thm:lpm_quotients}. Suppose that this is not the case, and assume that $(\ell_{i_s},u_{j_s})$ is the first bad pair in this list. Following Definition~\ref{def:good_pair} there are two cases to consider that make $(\ell_{i_s},u_{j_s})$ bad.
 
\emph{Case 1}: If the step $\ell_{i_s}$ is above the step $u_{j_s}$, i.e., $j_s<i_s$, consider the row-interval $[\ell_{j_s},u_{j_s}]$ in $M$. By the choice of $(\ell_{i_s},u_{j_s})$ it follows that $\ell_{j_s}\in L'$. Hence, in $M'$ there is no column-interval beginning with $\ell_{j_s}$. Thus, in order to represent $[\ell_{j_s},u_{j_s}]$ as union of intervals in $M'$, the row-interval of $M'$ starting with $\ell_{j_s}=:\ell'_j\in L$ has to be used. This interval is the interval $[\ell'_j,u'_j]$. But then, since $u_{j_s}\notin U'$, we have that $u'_j>u_{j_s}$ and thus $[\ell'_j,u'_j]$ contains properly the interval $[\ell_{j_s},u_{j_s}]$. This contradicts the fact that $[\ell_{j_s},u_{j_s}]$ is union of intervals in $M'$.
 
\emph{Case 2}: If the step $\ell_{i_s}$ is to the left of the step $u_{j_s}$. Let $\overline{\ell}$ be the smallest element in $\overline{L}$ larger than $\ell_{i_s}$. Graphically, $\overline{\ell}$ is the first east step after $\ell_{i_s}$ in the southern boundary of the diagram of $M$. Thus $\overline{\ell}$ determines the column-interval $[\overline{\ell},\overline{u}]$ in $M$. In order to express this interval as union of intervals in $M'$, Observation~\ref{obs:columnsbecolumns} tells us that only column-intervals in $M'$ can be used. In particular, since $\overline{\ell}\in \overline{L'}$, the column-interval $[\overline{\ell}, {\overline u'}]$ of $M'$ has to be used, where $\overline{u'}\in\overline U'$. However, $\overline {u'}>\overline u$ as $\overline{\ell}>\ell_{i_s}$ and the step $\ell_{i_s}$ becomes horizontal in $M'$ making the containment $[\overline{\ell},\overline{u}]\subsetneq[\overline{\ell},\overline {u'}]$ proper.  As in Case 1, this contradicts the fact that $[\overline{\ell},\overline{u}]$ is a union of intervals in $M'$.

Thus we conclude that $M'$ is a quotient of $M$.
 
 \noindent (i) $\Rightarrow$ (iii): Let $M'\leq_Q M$. It is sufficient to assume that $\rank(M')=\rank(M)-1$. Hence, by Theorem~\ref{thm:lpm_quotients} there is a good pair $(\ell,u)$ such that $U'=U\setminus\{ u\}$ and $L'=L\setminus\{ \ell\}$. Let $[\overline{\ell},\overline{u}]$ be a column-interval of $M$ and let us prove that it can be written as union of intervals in $M'$. Since the pair $(\ell,u)$ is good, in the diagram of $M'$, steps $\ell$ and $u$ become horizontal and thus the horizontal step $\overline{u}$ appears weakly to the right of $\overline{\ell}$ in $M'$. Hence, we can write the interval  $[\overline{\ell},\overline{u}]$ as union of column-intervals $\bigcup_{s=1}^t[\overline{\ell'}_{i_s},\overline{u'}_{i_s}]$ in $M'$ in such a way that $\overline{\ell'}_{i_1}=\overline{\ell}$ and $\overline{u'}_{i_t}=\overline{u}$. In this way, every column-interval of $M$ can be written as required in $M'$. See the dark grey interval in Figure~\ref{fig:intervals}.

\begin{figure}[htp]
    \centering
    \includegraphics[width=.9\textwidth]{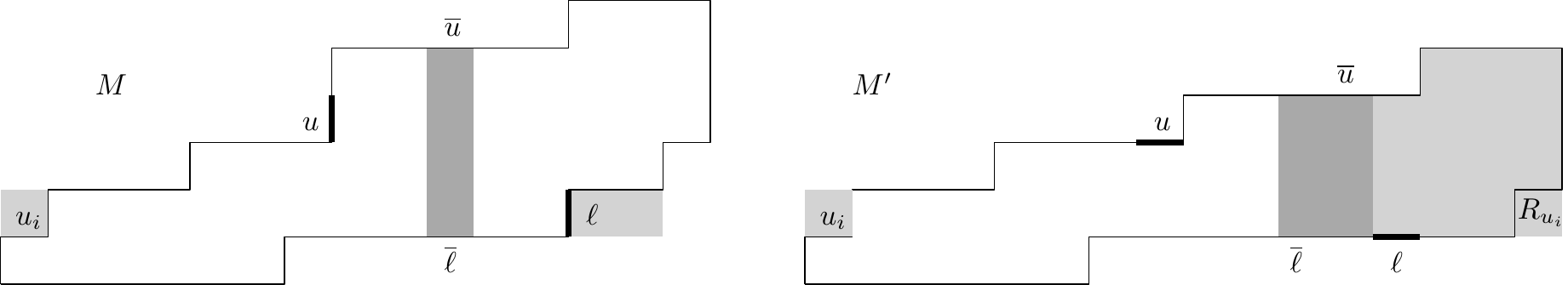}
    \caption{Representing an interval of $M$ as union of intervals of $M'$.}\label{fig:intervals}
\end{figure}

Now, consider a row-interval $[\ell_i,u_i]$ in $M$. If $\ell_i=\ell$ then $\ell_i\in\overline {L'}$ and we take the column-intervals in $M'$ of the form $[\overline {\ell'}, \overline {u'}]$ where $\overline {\ell'}\in\overline {L'}$ and $\overline{\ell'}\geq\ell_i$. Consider the union of these intervals and, if $u_i\neq u$, further join the unique row-interval $R_{u_i}$ of $M'$ with end-point $u'$. This yields $[\ell_i,u_i]$. See the light grey interval in Figure~\ref{fig:intervals}.
Reasoning in an analogous way, if $u_i=u$ we obtain $[\ell_i,u_i]$ as union of the column-intervals in $M'$ of the form $[\overline {\ell'}, \overline {u'}]$ where $\overline {u'}\in\overline {U'}$ and $\overline {u'}\leq u_i$, along with the unique  row-interval $R_{\ell_i}$ of $M'$ whose initial point is $\ell_i$. If $\ell'\neq\ell$ and $u'\neq u$, then we take in $M'$ the union of the row-intervals $R_{\ell_i}\cup R_{u_i}$.
The result follows.
 \end{proof}

 \section{Further remarks}

\subsection{Properties of $\mathcal P_n$}
We have already explored some properties of the poset $\mathcal P_n$. 
Often, the techniques developed to answer enumerative properties of a poset like $\mathcal P_n$  lead to unforseen connections in mathematics. 
Hence we are interested in the following questions.
\begin{question}\label{quest:unimodal}
   Are rank functions of intervals of $\mathcal P_n$ unimodal?                                \end{question}
Note that through Corollary~\ref{cor:unimodular} we know that the answer is positive for the entire poset $\mathcal P_n$. Theorem~\ref{thm:lpm_quotients} together with Lemma~\ref{lem:removing} shed some further light on the structure of the order complex of an interval $[M',M]_Q$ in $\mathcal P_n$. The idea is that if $((\ell_{i_1},u_{j_1}), \ldots, (\ell_{i_z},u_{j_z}))$ is the greedy pairing on  $(U\setminus U',L\setminus L')$ then any permutation of the set of pairs $\{(\ell_{i_1},u_{j_1}), \ldots, (\ell_{i_z},u_{j_z})\}$ gives rise to a sequence that is a good pairing. That is, every such permutation corresponds to a maximal chain in the interval $[M',M]_Q$.  However, not all maximal chains arise this way. For instance the interval $[U_{1,3},U_{3,3}]_Q$ in $\mathcal P_3$ has 3 maximal chains, two of which come as permutations of the set $\{(1,2),(2,3) \}$, The third chain corresponds to the sequence $((1,3), (2,2))$. Notice that $((2,2), (1,3))$ is not a good pairing on $(12,23)$ as $(1,3)$ is not a good pair of $M=M[13,13]$. See Figure~\ref{fig:intervalP3}.  Along these lines, a better understanding of maximal chains in $\mathcal P_n$ could allow us to understand and explore shellability and Whitney duality, as defined in~\cite{GH}.

\begin{question}
    Is $\mathcal P_n$ shellable or does it have a Whitney dual?
\end{question}

\begin{figure}[htp]
    \centering
    \includegraphics[width=.5\textwidth]{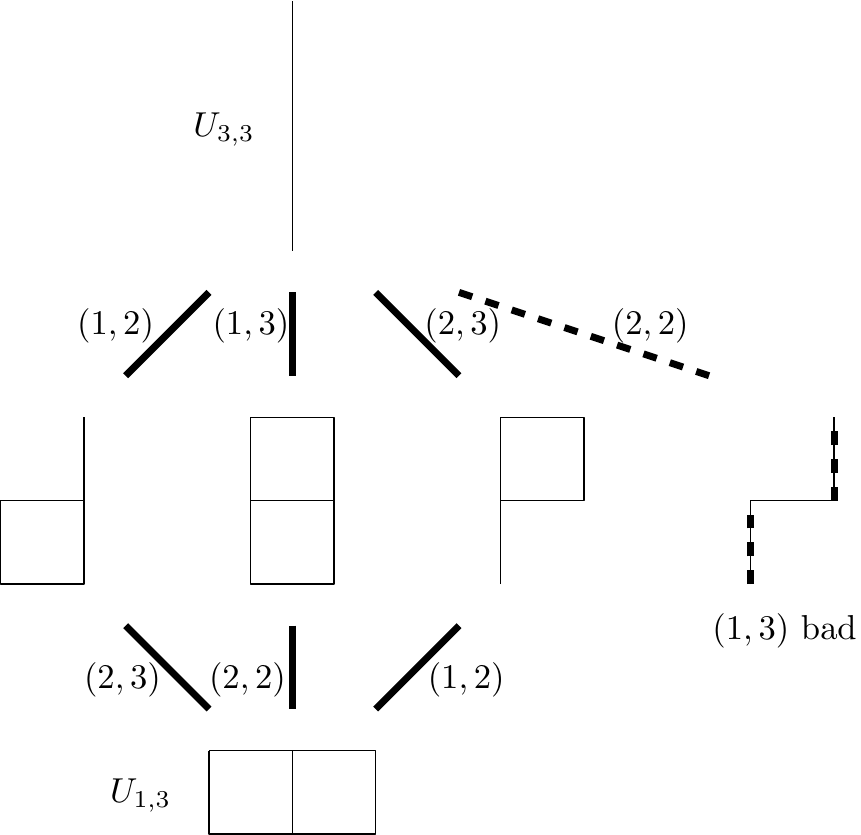}
    \caption{The interval $[U_{1,3},U_{3,3}]_Q$ in $\mathcal P_3$.}\label{fig:intervalP3}
\end{figure}

\subsection{Towards LPM flag diagrams}\label{sec:diagrams}
In~\cite{DeM-07} a certain class of (partial) LPFMs was studied, i.e., $\mathcal{F}:(M_0, M_1, \ldots, M_k)$ such that $U_{0,n}=M_0\leq_Q\ldots\leq_Q M_k=U_{n,n}$ where all components are LPMs and $k\leq n$. Given a flag of bases $\mathfrak{B}=(B_0, B_1, \ldots, B_k)$ in $\mathcal{F}$,  one can associate a monotone path $P$ of length $n$ in $\mathbb{Z}^{k}$ by setting the $i$th step to $e_j$ if $i\in B_j\setminus B_{j-1}$ for all $1\leq i\leq k$. Note that if $\mathcal{F}=(U_{0,n}, M, U_{n,n})$ where  $M=M[U,L]$, then the set of paths obtained this way just corresponds to the paths in the diagram of $M[U,L]$. It is thus natural to define 
the \emph{diagram} $D_{\mathcal{F}}$ of $\mathcal{F}$ as the set of points in $\mathbb{Z}^{k}$ that are on paths associated to flags of bases of $\mathcal{F}$. See Figure~\ref{fig:anna} for an example.
\begin{problem}
(a) Characterize the set of diagrams of LPFMs. (b) Characterize those paths in a diagram that correspond to flags. Are these all the monotone ones?
\end{problem}
This question is already present in~\cite[Figure 6]{DeM-07}, where an example shows that already pretty reasonable sets in $\mathbb{Z}^3$ are not the diagram of an LPFM. We hope that the results of the present paper allow to shed new light on this problem.

\begin{figure}[htp]
    \centering
    \includegraphics[width=.5\textwidth]{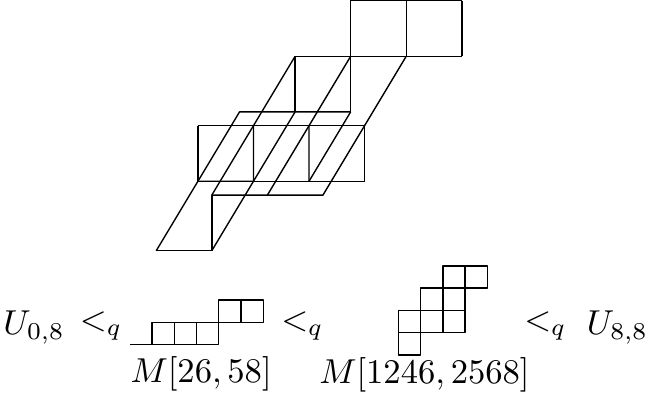}
    \caption{An LPFM and its diagram.}\label{fig:anna}
\end{figure}

{

\subsection{Weak order and Higgs lift}\label{sec:Higgs}

Let $\mathcal M_{k,n}$ be the collection of matroids over the set $[n]$ of fixed rank $k$. This collection is endowed with a partial ordering $\leq_W$, known as the \emph{(rank-preserving) weak order} given as follows: if $M',M\in\mathcal M_{k,n}$ then $M\leq_W M'$ if and only if  every basis of $M'$ is a basis of $M$. See \cite[Prop. 7.4.7]{Bry-86} for several cryptomorphic descriptions of the rank-preserving weak order relation.

In the case of LPMs, the weak order corresponds to diagram containment. That is, if $\mathcal{L}_{k,n}$ denotes the set of LPMs of rank $k$ over $[n]$ and  $M'=M[U',L'], M=M[U,L]\in \mathcal{L}_{k,n}$ then $M'\leq_W M$ if and only if $U'\geq_G U$ and  $L'\leq_G L$. 

Since $({[n] \choose k}, \leq_G)$ has a lattice structure, by Observation~\ref{obs:weak} we have that $(\mathcal{L}_{k,n},\leq_W)$ becomes an upper semilattice by setting the join $M[U,L]\vee M[U',L']:=M[U\wedge_G U',L\vee_G L']$. In particular, since $\leq_G$ is a distributive lattice, intervals in $(\mathcal{L}_{r,n},\leq_W)$ are distributive lattices, as well. Also maxima and minima are easily determined as we now state.

\begin{observation}\label{obs:orderonLPMS}
The poset $(\mathcal{L}_{k,n},\leq_W)$ is isomorphic to the upper semilattice of intervals of the Gale order $({[n] \choose k}, \leq_G)$ ordered by inclusion. Its unique maximum is $U_{k,n}$. It has ${n\choose k}$ minima corresponding to the elements of ${[n] \choose k}$. 
\end{observation}

One can wonder how the weak order and the quotient relation interact.
In Figure~\ref{fig:quotients} we illustrate all the LPMs that belong to the interval $[U_{0,8},M]_Q$, where $M=M[1246,2568]$. That is,  $N\in[U_{0,8},M]_Q$ if and only if $N$ is an LPM and $N\leq_Q M$. Matroids in this interval that have the same rank have been ordered using $\leq_W$. Notice also that although $M[12,58] <_W M[12,68]$ and $M[12,68]\leq_Q M[124,268]$, it does not follow that $M[12,58]$ is a quotient of $M[124,268]$. Thus, the union of quotient relation and rank preserving weak order is not an order relation.

Given two matroids $M'$ and $M$ such that $M'\leq_Q M$, we say that a matroid $N$ is the $i$th \emph{Higgs lift} of $M'$ towards $M$ if $N$ is the maximal matroid (with respect to $\leq_W$) such that $r(N)=r(M')+i$ and $M'\leq_Q N\leq_Q M$. See ~\cite[Propositions 2.2, 2.6]{BS11} and~\cite{BCN21} for the proof that the Higgs lift always exists. Notice that the $i$th {Higgs lift} of $U_{0,n}$ towards $M$ is simply the $r-i$-truncation of $M$ if $M$ has rank $r$.
With the above notation one can now wonder if a given class of matroids $\mathcal C$ is closed under taking Higgs lifts, where we would generalize the notion in the following sense. That is, if $M'\leq _q M$ are in $\mathcal C$ and $i\leq r(M)-r(M')$, then $N\in \mathcal{C}$ is a \emph{Higgs lift} of $M',M$ if $N$ is the unique maximal (with respect to $\leq_W$) $N\in\mathcal C$  such that $r(N)=r(M')+i$ and $M'\leq_Q N\leq_Q M$. However, in general there exists no Higgs lift within the class of LPMs: going back to Figure~\ref{fig:quotients}, we see that there is no unique maximum with respect to $ \leq_W$ among the rank $3$ LPMs in the interval $[U_{0,8},M[1246,2568]]$.

\begin{figure}[htp]
    \centering
    \includegraphics[width=\textwidth]{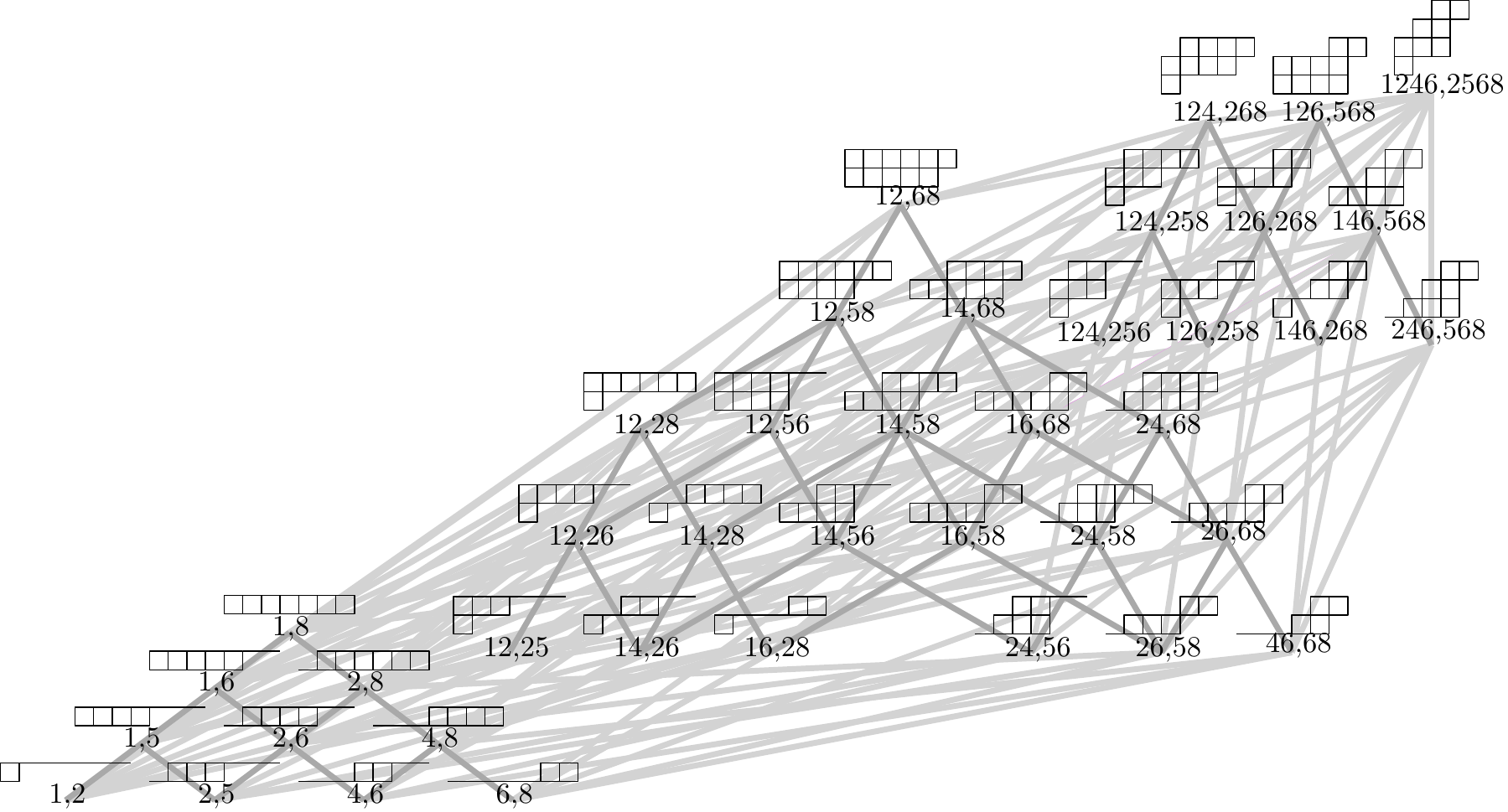}
    \caption{The LPM $M[1246,2568]$ and its (non-trivial) quotients. Each rank is ordered with respect to the weak order, the corresponding cover relations are in dark gray. Quotient cover relations are in light gray.}\label{fig:quotients}
\end{figure}

\begin{observation}
 The class of LPMs is not closed under Higgs lifts.
\end{observation}

Recall that another question that remains open about the different ranks of the quotient order of LPMs is whether they are unimodal on any interval $[M', M]_Q$ (see Question~\ref{quest:unimodal}). Note that we have answered this in the positive for the entire poset $\mathcal{P}_n$ itself.

}

\subsubsection*{Acknowledgments.}
We thank Lauren Williams for insightful conversations and Suho Oh for sharing their conjecture with us.
Benedetti was supported by grant FAPA of the Faculty of Science at Universidad de los Andes.
Knauer was  supported by the Spanish State Research Agency
through grants RYC-2017-22701, PID2019-104844GB-I00 and the Severo Ochoa and María de Maeztu Program for Centers and Units of Excellence in R\&D (CEX2020-001084-M).
 We also thank the CIMPA-ICTP Research In Pairs program.

%
%

%
%

%
%
%

\bibliography{bibliografia}

\providecommand{\bysame}{\leavevmode\hbox to3em{\hrulefill}\thinspace}
\providecommand{\MR}{\relax\ifhmode\unskip\space\fi MR }
\providecommand{\MRhref}[2]{%
  \href{http://www.ams.org/mathscinet-getitem?mr=#1}{#2}
}
\providecommand{\href}[2]{#2}
\begin{thebibliography}{KMSA18}

\bibitem[AJK20]{An-17}
Suhyung {An}, JiYoon {Jung}, and Sangwook {Kim}, \emph{{Facial structures of
  lattice path matroid polytopes}}, {Discrete Math.} \textbf{343} (2020),
  no.~1, 11, Id/No 111628.

\bibitem[ARW17]{ARW17}
Federico {Ardila}, Felipe {Rinc\'on}, and Lauren {Williams}, \emph{{Positively
  oriented matroids are realizable}}, {J. Eur. Math. Soc. (JEMS)} \textbf{19}
  (2017), no.~3, 815--833.

\bibitem[BCN21]{BCN21}
Joseph~E. {Bonin}, Carolyn {Chun}, and Steven~D. {Noble}, \emph{{Delta-matroids
  as subsystems of sequences of Higgs lifts}}, {Adv. Appl. Math.} \textbf{126}
  (2021), 27, Id/No 101910.

\bibitem[BCT22]{BCT}
Carolina {Benedetti}, Anastasia {Chavez}, and Daniel {Tamayo}, \emph{{Quotients
  of uniform matroids}}, {Elec. J. Comb.} \textbf{29} (2022), no.~1.

\bibitem[Bd06]{Bon-06}
Joseph~E. {Bonin} and Anna {de Mier}, \emph{{Lattice path matroids: structural
  properties}}, {Eur. J. Comb.} \textbf{27} (2006), no.~5, 701--738.

\bibitem[BdMN03]{Bon-03}
Joseph~E. Bonin, Anna de~Mier, and Marc Noy, \emph{Lattice path matroids:
  enumerative aspects and {T}utte polynomials}, J. Combin. Theory Ser. A
  \textbf{104} (2003), no.~1, 63--94.

\bibitem[BEW22]{BEW22}
Jonathan Boretsky, Christopher Eur, and Lauren Williams, \emph{Polyhedral and
  tropical geometry of flag positroids}, arXiv:2208.09131 (2022).

\bibitem[BG07]{Bon-07}
Joseph~E. Bonin and Omer Gim{\'e}nez, \emph{Multi-path matroids}, Combin.
  Probab. Comput. \textbf{16} (2007), no.~2, 193--217.

\bibitem[BGW03]{BGW}
Alexandre~V. Borovik, Israel~M. Gel'fand, and Neil White, \emph{{Coxeter
  Matroids}}, vol. 216, London: Birkh\"auser, 2003.

\bibitem[{Bid}12]{Bid-12}
Hoda {Bidkhori}, \emph{{Lattice Path Matroid Polytopes}}, arXiv:1212.5705
  (2012).

\bibitem[BKVP23]{BKV23}
Carolina Benedetti, Kolja Knauer, and Jerónimo Valencia-Porras, \emph{On
  lattice path matroid polytopes: alcoved triangulations and snake
  decompositions}, arXiv:2303.10458 (2023).

\bibitem[{Blu}01]{B07}
Stefan {Blum}, \emph{{Base-sortable matroids and Koszulness of semigroup
  rings}}, {Eur. J. Comb.} \textbf{22} (2001), no.~7, 937--951.

\bibitem[Bon10]{Bon-10}
Joseph~E. Bonin, \emph{Lattice path matroids: the excluded minors}, J. Combin.
  Theory Ser. B \textbf{100} (2010), no.~6, 585--599.

\bibitem[Bor22]{Bor22}
Jonathan Boretsky, \emph{Positive tropical flags and the positive tropical
  {Dressian}}, S{\'e}min. Lothar. Comb. \textbf{86B} (2022), 12 (English),
  Id/No 86.

\bibitem[{Bry}86]{Bry-86}
Thomas {Brylawski}, \emph{{Constructions}}, {Theory of matroids, Encycl. Math.
  Appl. 26, 127-223 (1986).}, 1986.

\bibitem[BS11]{BS11}
Joseph~E. {Bonin} and William~R. {Schmitt}, \emph{{Splicing matroids}}, {Eur.
  J. Comb.} \textbf{32} (2011), no.~6, 722--744.

\bibitem[BS22]{black2022flag}
Alexander~E. Black and Raman Sanyal, \emph{Flag polymatroids}, arXiv:2207.12221
  (2022).

\bibitem[CDMS22]{am2018flag}
Amanda Cameron, Rodica Dinu, Mateusz Micha{\l}ek, and Tim Seynnaeve, \emph{Flag
  matroids: algebra and geometry}, Interactions with lattice polytopes.
  Selected papers based on the presentations at the workshop, Magdeburg,
  Germany, September 14--16, 2017, Cham: Springer, 2022, pp.~73--114 (English).

\bibitem[CO22]{corey2022initial}
Daniel Corey and Jorge~Alberto Olarte, \emph{Initial degenerations of flag
  varieties}, arXiv:2207.08094 (2022).

\bibitem[CRA11]{Cha-11}
Vanessa Chatelain and Jorge~Luis Ram{\'{\i}}rez~Alfons{\'{\i}}n, \emph{Matroid
  base polytope decomposition}, Adv. in Appl. Math. \textbf{47} (2011), no.~1,
  158--172.

\bibitem[DD15]{Del-12}
Emanuele {Delucchi} and Martin {Dlugosch}, \emph{{Bergman complexes of lattice
  path matroids.}}, {SIAM J. Discrete Math.} \textbf{29} (2015), no.~4,
  1916--1930.

\bibitem[{de }07]{DeM-07}
Anna {de Mier}, \emph{{A natural family of flag matroids}}, {SIAM J. Discrete
  Math.} \textbf{21} (2007), no.~1, 130--140.

\bibitem[dS87]{dS87}
Ilda~P.F. da~Silva (ed.), \emph{{Quelques propri\'et\'es des matroides
  orient\'es}}, Ph.D. Dissertation, Universit\'e Paris VI, 1987.

\bibitem[EHL23]{EHL23}
Christopher Eur, June Huh, and Matt Larson, \emph{Stellahedral geometry of
  matroids}, arXiv:2207.10605 (2023).

\bibitem[Fer22]{Fer22}
Luis Ferroni, \emph{On the {Ehrhart} polynomial of minimal matroids}, Discrete
  Comput. Geom. \textbf{68} (2022), no.~1, 255--273.

\bibitem[FJS22]{FJS22}
Luis Ferroni, Katharina Jochemko, and Benjamin Schr{\"o}ter, \emph{Ehrhart
  polynomials of rank two matroids}, Adv. Appl. Math. \textbf{141} (2022), 26,
  Id/No 102410.

\bibitem[FS22]{BS22}
Luis Ferroni and Benjamin Schröter, \emph{Valuative invariants for large
  classes of matroids}, arXiv:2208.04893 (2022).

\bibitem[GH21]{GH}
R.~{Gonzalez} and J.~{Hallam}, \emph{{The Whitney duals of a graded poset}},
  {JCTA} \textbf{177} (2021).

\bibitem[{Hig}68]{H68}
Denis~A. {Higgs}, \emph{{Strong maps of geometries}}, {J. Comb. Theory}
  \textbf{5} (1968), 185--191.

\bibitem[HP18]{HP18}
Chris {Heunen} and Vaia {Patta}, \emph{{The category of matroids}}, {Appl.
  Categ. Struct.} \textbf{26} (2018), no.~2, 205--237.

\bibitem[JL22]{jarra2022flag}
Manoel Jarra and Oliver Lorscheid, \emph{Flag matroids with coefficients},
  arXiv:2204.04658 (2022).

\bibitem[JLLO23]{Joswig_2023}
Michael Joswig, Georg Loho, Dante Luber, and Jorge~Alberto Olarte,
  \emph{Generalized permutahedra and positive flag dressians}, International
  Mathematics Research Notices (2023).

\bibitem[KMR18]{Kna-18}
Kolja {Knauer}, Leonardo {Mart\'{\i}nez-Sandoval}, and Jorge~Luis {Ram\'{\i}rez
  Alfons\'{\i}n}, \emph{{On lattice path matroid polytopes: integer points and
  Ehrhart polynomial}}, {Discrete Comput. Geom.} \textbf{60} (2018), no.~3,
  698--719.

\bibitem[KMSA18]{KMR}
Kolja Knauer, Leonardo Martínez-Sandoval, and Jorge Luis~Ramírez Alfonsín,
  \emph{A {T}utte polynomial inequality for lattice path matroids}, Advances in
  Applied Mathematics \textbf{94} (2018), no.~Supplement C, 23 -- 38, Special
  issue on the Tutte polynomial.

\bibitem[Kra15]{Kra15}
Christian Krattenthaler, \emph{Lattice path enumeration}, Handbook of
  enumerative combinatorics, Boca Raton, FL: CRC Press, 2015, pp.~589--678.

\bibitem[Kun86]{Kun86}
Joseph P.~S. Kung, \emph{Strong maps}, Theory of matroids, {Encycl}. {Math}.
  {Appl}. 26, 224-253 (1986)., 1986.

\bibitem[KW15]{KW15}
Yuji Kodama and Lauren Williams, \emph{The full {Kostant}-{Toda} hierarchy on
  the positive flag variety}, Commun. Math. Phys. \textbf{335} (2015), no.~1,
  247--283.

\bibitem[MSD19]{MOX19}
Robert {Mcalmon}, Oh~{Suho}, and Xiang {David}, \emph{{Flats of a positroid
  from its decorated permutation}}, {S\'emin. Lothar. Comb.} \textbf{82B}
  (2019), 82b.43, 12.

\bibitem[MT15]{Mor-13}
Jason {Morton} and Jacob {Turner}, \emph{{Computing the Tutte polynomial of
  lattice path matroids using determinantal circuits.}}, {Theor. Comput. Sci.}
  \textbf{598} (2015), 150--156.

\bibitem[{Nel}18]{N18}
Peter {Nelson}, \emph{{Almost all matroids are nonrepresentable}}, {Bull. Lond.
  Math. Soc.} \textbf{50} (2018), no.~2, 245--248.

\bibitem[NK09]{NK09}
Hirokazu {Nishimura} and Susumu {Kuroda} (eds.), \emph{{A lost mathematician,
  Takeo Nakasawa. The forgotten father of matroid theory}}, Basel:
  Birkh\"auser, 2009.

\bibitem[OX22]{OX22}
SuHo Oh and David Xiang, \emph{The facets of the matroid polytope and the
  independent set polytope of a positroid}, J. Comb. \textbf{13} (2022), no.~4,
  545--560.

\bibitem[{Oxl}11]{O11}
James~G. {Oxley}, \emph{{Matroid theory}}, vol.~21, Oxford: Oxford University
  Press, 2011.

\bibitem[Pad23]{padro2023efficient}
Carles Padró, \emph{Efficient representation of lattice path matroids},
  arXiv:2310.10489 (2023).

\bibitem[Pos06]{P06}
Alexander Postnikov, \emph{{Total positivity, Grassmannians, and networks}},
  {arXiv:0609764} (2006).

\bibitem[Sch10]{Sch-10}
Jay Schweig, \emph{On the {$h$}-vector of a lattice path matroid}, Electron. J.
  Combin. \textbf{17} (2010), no.~1, Note 3, 6.

\bibitem[Sch11]{Sch-11}
\bysame, \emph{Toric ideals of lattice path matroids and polymatroids}, J. Pure
  Appl. Algebra \textbf{215} (2011), no.~11, 2660--2665.

\bibitem[Sta01]{EC2}
Richard Stanley, \emph{{Enumerative Combinatorics Vol. 2}}, paperback ed. ed.,
  vol.~62, Cambridge: Cambridge University Press, 2001.

\bibitem[TW15]{TW15}
Emmanuel {Tsukerman} and Lauren {Williams}, \emph{{Bruhat interval polytopes}},
  {Adv. Math.} \textbf{285} (2015), 766--810.

\bibitem[{Whi}34]{W34}
Hassler {Whitney}, \emph{{On the abstract properties of linear dependence.}},
  {Bull. Am. Math. Soc.} \textbf{40} (1934), 663.

\end{thebibliography}
\bibliographystyle{amsalpha}

\end{document}